\documentclass[12pt,letterpaper]{article}
\usepackage{natbib}

\usepackage{bbm}
\usepackage{bm}
\usepackage{chngcntr}
\usepackage{multirow}

\usepackage[colorlinks=true,urlcolor=blue,citecolor=blue,linkcolor=blue,bookmarks=true,bookmarksopen=false]{hyperref}

\setcitestyle{round,aysep={,},yysep={;}, citesep={;}}
\usepackage[ruled,linesnumbered,vlined]{algorithm2e}
\usepackage{algorithmic}
\usepackage{booktabs}
\usepackage{hhline}

\newcommand{\vc}{\mathbf{c}}

\newcommand{\vX}{\mathbf{X}}

\newcommand{\vh}{\mathbf{h}}

\newcommand{\vq}{\mathbf{q}}
\newcommand{\vd}{\mathbf{d}}
\newcommand{\vw}{\mathbf{w}}
\newcommand{\vR}{\mathbf{R}}
\newcommand{\vv}{\mathbf{v}} 
\newcommand{\vz}{\mathbf{z}} 
\newcommand{\vZ}{\mathbf{Z}} 
\newcommand{\vx}{\mathbf{x}} 
\newcommand{\vy}{\mathbf{y}} 

\newcommand{\veta}{\bm{\eta}} 
\newcommand{\vpi}{\bm{\pi}} 
\newcommand{\vlambda}{\bm{\lambda}} 

\newcommand{\vmu}{\bm{\mu}} 
\newcommand{\vnu}{\bm{\nu}} 
\newcommand{\vb}{\mathbf{b}} 
\newcommand{\vf}{\mathbf{f}}

\newcommand{\vbeta}{\bm{\beta}}

\newcommand{\vu}{\mathbf{u}}

\newcommand{\vtheta}{\bm{\theta}}

\newcommand{\vxi}{\bm{\xi}}
\newcommand{\vgamma}{\bm{\gamma}}

\renewcommand{\Omega}{\varOmega}
\newcommand{\R}{\mathbbm{R}}
\newcommand{\E}{\mathbbm{E}}

\newcommand{\M}{\mathcal{M}}
\newcommand{\cL}{\mathcal{L}}

\newcommand{\X}{\mathcal{X}}
\newcommand{\Y}{\mathcal{Y}}

\newcommand{\C}{\mathcal{C}}
\newcommand{\dd}{\mathrm{d}}
\newcommand{\cvar}{\operatorname{CVaR}}
\newcommand{\var}{\operatorname{VaR}}

\newcommand{\MP}{\operatorname{MP}}

\newcommand{\vg}{\mathbf{g}}
\newcommand{\vbg}{\mathbf{\bar{g}}}
\newcommand{\vtg}{\mathbf{\tilde{g}}}
\newcommand{\vG}{\mathbf{G}}

\newcommand{\supp}{\operatorname{support}}
\newcommand{\OBF}{\operatorname{OBF}}
\newcommand{\LP}{\operatorname{LinearP}}
\newcommand{\LD}{\operatorname{LinearD}}

\newcounter{commentcounter}
\long\def\symbolfootnote[#1]#2{\begingroup%
\def\thefootnote{\fnsymbol{footnote}}\footnote[#1]{#2}\endgroup}
\newcommand{\comment}[1]{{\footnotesize\textbf{\textcolor{red}{(C.\arabic{commentcounter})}}\symbolfootnote[4]{\texttt{\textcolor{red}
        {(C.\arabic{commentcounter})~#1}}}}\addtocounter{commentcounter}{1}}

\usepackage{sibarticle}
\usepackage{threeparttable}
\usepackage{pdflscape}

\DeclareMathOperator*{\argmin}{arg\,min}
\allowdisplaybreaks
\newcommand{\ignore}[1]{}

\title{Two-stage Stochastic Programming under Multivariate Risk Constraints with an Application to Humanitarian Relief Network Design}
\ShortTitle{Two-Stage optimization with multivariate risk
constraints} \ShortAuthors{ Merakl{\i}, Noyan, K\"{u}\c{c}\"{u}kyavuz} \NumberOfAuthors{2}
\FirstAuthor{ Nilay
Noyan} \FirstAuthorAddress {Industrial Engineering Program,
Sabanc{\i} University, Turkey\\ {\tt nnoyan@sabanciuniv.edu}}
\SecondAuthor{Merve Merakl{\i}, Simge K\"{u}\c{c}\"{u}kyavuz}
\SecondAuthorAddress{Industrial and Systems Engineering, University
of Washington, U.S.A.\\ {\tt merakli@uw.edu, simge@uw.edu}}

\keywords{stochastic programming; multicriteria optimization;
risk-averse two-stage; multivariate risk; conditional value-at-risk;
Benders decomposition; branch-and-cut; network design; pre-disaster;
humanitarian relief}

\begin{document}

\maketitle \centerline{January 10, 2017}

\begin{abstract}In this study, we consider two classes of multicriteria
two-stage stochastic programs in finite probability spaces with
multivariate risk constraints. The first-stage problem features a
multivariate stochastic benchmarking constraint based on a
vector-valued random variable representing multiple and possibly
conflicting stochastic performance measures associated with the
second-stage decisions. In particular, the aim is to ensure that the
associated random outcome vector of interest is preferable to a
specified benchmark with respect to the multivariate polyhedral
conditional value-at-risk (CVaR) or a multivariate stochastic order
relation. In this case, the classical decomposition methods cannot
be used directly due to the complicating multivariate stochastic
benchmarking constraints. We propose an exact unified decomposition
framework for solving these two classes of optimization problems and
show its finite convergence. We apply the proposed approach to a
stochastic network design problem in a pre-disaster humanitarian
logistics context and conduct a computational study concerning the
threat of hurricanes in the Southeastern part of the United States.
Our numerical results on these large-scale problems show that our
proposed algorithm is computationally scalable.

\ignore{our proposed algorithm tam precise degil ama abstract boyle
kalsa da olur...iki algoritma var ve asil SD-based olan scalable}
\end{abstract}

\section{Introduction}
Two-stage stochastic programming is a vibrant research area, which
provides a natural and widely applicable modeling framework for
decision making problems under uncertainty in a large variety of
fields. Such models are well-suited for the situations where
decisions are made in two stages; the first-stage decisions are made before
the uncertainty is revealed, and the second-stage (recourse)
decisions are made given the predetermined first-stage decisions and
the observed realization of the random parameters. For many
practical decision making problems, it is essential to evaluate the
decisions according to multiple and possibly conflicting stochastic
criteria. Along these lines, we focus on two-stage stochastic
programming models involving first-stage decisions leading to
uncertain outcomes that can be evaluated according to multiple
stochastic performance measures associated with the corresponding
second-stage decisions.

In multicriteria stochastic optimization, it is common to employ a
weighted-sum approach, which aggregates the specified multiple
stochastic objective criteria to obtain a single objective function.
Following this mainstream approach, we formulate the second-stage
problem as a multiobjective optimization problem and focus on its
expected optimal objective value in the first stage. In addition,
one may also prefer to impose constraints on some performance
measures associated with the second-stage decisions. To this end, we
employ stochastic benchmarking preference relations between the
vector-valued random outcomes of the second-stage decisions. In
particular, in the first class of problems we study, the first-stage
problem enforces the multivariate polyhedral CVaR constraint
introduced by \citet{Noyan13}. This approach considers a family of
linear scalarization functions (a polyhedral set of scalarization
weight vectors) and requires all scalarized versions of the
decision-based random vector of outcomes to be preferable to the
corresponding scalarizations of an existing reference (benchmark)
outcome according to the univariate CVaR relation. Furthermore,
allowing arbitrary polyhedra as scalarization sets also provides a
good balance between flexibility and computational tractability. To
the best of our knowledge, our study is a first in developing such a
risk-averse two-stage model featuring multivariate risk
measure-based constraints. We also note that our proposed modeling
framework is not limited to the risk-neutral objective function; it
can be easily extended to the case where the objective function
features also a risk measure representing the decision makers' risk
preferences regarding the random recourse cost. For acceptable
risk-averse objectives, which lead to computationally tractable
two-stage stochastic programming models, we refer the reader to
\cite{Ahmed06} and \cite{Dupacova08}.

Optimization problems with the multivariate stochastic benchmarking
constraints have been receiving increasing attention in the
literature. Using multivariate stochastic relations is essential to
capture the correlation between the multiple random outcomes. The
majority of the existing studies extend a univariate stochastic
preference relation, based on second-order stochastic dominance
(SSD) or a coherent risk measure such as CVaR, to the multivariate
case by considering a family of scalarization functions. While
multivariate SSD-constrained problems are more frequently studied
\citep[see, e.g.,][]{Dentcheva09,Mello09,Hu12,DW15,DW16}, there is an
increasing interest in risk measure-constrained variants \citep[see, e.g.,][]{Noyan13,Liu15}, which provide natural relaxations to overly
demanding and conservative SSD-based models. More recently, studies
focusing on both multivariate SSD- and risk measure-constrained
models also appear in literature \citep{Kucukyavuz14,Noyan16}. In
these existing studies, the scalarization functions are almost
exclusively of the linear type with the exception of \citet{Noyan16}
who extend the multivariate risk-constrained models by incorporating
a general class of scalarization functions (including a variety of
non-linear scalarization functions). Moreover, we point out that the
scalarization-based line of research has been dedicated almost
entirely to single-stage (static) decision-making problems. The only
exception is the work of \citet{DW16}, which introduces a two-stage
stochastic optimization problem involving a multivariate SSD (more
precisely, referred to as the increasing convex ordering in the
context of cost minimization) constraint on a random outcome vector
of the second-stage decisions with respect to the unit simplex of
scalarization vectors. \ignore{Their stochastic preference relation
based on the unit simplex of scalarization vectors is equivalent to
the positive linear SSD relation, which enforces a univariate SSD
constraint to all nonnegative weighted combinations of random
outcomes.} The authors present two approximate decomposition-based
solution algorithms which rely on the Lagrangian relaxation of the
multivariate stochastic ordering constraint and show their finite
convergence to an $\varepsilon$-feasible $\varepsilon$-optimal
solution, even if the probability space is not finite; their results
and algorithms are adaptations of those presented in their earlier
study \citep{DW15} for the single-stage case. We contribute to this
line of research by introducing an alternative two-stage model with
multivariate CVaR constraints. In addition, we consider the
SSD-based counterpart (as in \citealp{DW15}) and
provide a new computationally tractable and exact solution
algorithm for this problem class. We defer a detailed discussion on the advantages of our
solution method compared to the existing approaches for the
multivariate SSD-constrained two-stage models to Section
\ref{sec:ssd}.

\ignore{``Following this line of research, for the second class of
problems we study, we show that our proposed solution framework is
immediately applicable to the two-stage model featuring multivariate
SSD constraints''...Bu sonradan soyleniyor zaten, proposed solution
method'u anlatmadan it is also applicable to this case dememek daha
iyi olabilir... cozum yontemi gelistirdik deyip ve contribution'in
altini cizmek yeterli bence...alternative deyince CVaR onceki
altatima da baglaniyor...literature review sonunda iki
contribution'u listelemis oldum...boylece discussion biterken
contribution da highlight edilmis oluyor, yeni verisyonda SSD kismi
icin bu contribution cumleleri yoktu..intoduction'da
advantages'larinin oldugunu not etmek de iyi olabilir...bunlar
contribution acisindan daha onemli notlar...ilgili cumle, tam
existing'ten bahsedilen kisima da uyuyor bence...}

Although, it is not directly related to the multicriteria stochastic
optimization problems of our interest, we note that there are a few
studies on risk-averse two-stage models with univariate (first- or
second-order) stochastic dominance and CVaR-based constraints \citep[see, e.g.,][]{Fabian08,Gollmer11,Dentcheva2012}. However, they employ
the risk constraints to compare scalar-based random variables. Our
study extends such risk-averse modeling approaches to the
multicriteria case, allowing us to consider additional stochastic
criteria associated with the second-stage decisions other than the
random recourse cost.

The proposed modeling approach with the multivariate risk
constraints is a fairly recent research area, and it has promise to
be applied in a wide variety of fields. This approach is
particularly well-suited for the field of humanitarian logistics,
since incorporating risk is crucial for rarely occurring disaster
events \citep[e.g.,][]{Noyan12}, and considering multiple conflicting
performance criteria (such as efficiency and equity) is often
essential for the effectiveness of the relief response systems
\citep[e.g.,][]{Vitoriano11,Huang2012,Gutjahr16}. Motivated by the
significance of the long-term pre-disaster planning, we apply the
proposed framework to a stochastic pre-disaster relief network
design problem. Hence, our study also contributes to the
humanitarian relief literature by introducing a new risk-averse
two-stage optimization model, which provides a flexible and
tractable way of considering decision makers' risk preferences based
on multiple stochastic criteria.

Next we summarize our contributions and give an outline of our
paper. In Section \ref{sec:general-model}, we describe a novel
two-stage stochastic program with multivariate CVaR constraints. It
is well-known that risk-neutral two-stage stochastic programs are
generally hard to solve due to a potentially large number of
scenario-dependent recourse decisions. Thus, introducing a
multivariate stochastic benchmarking relation, which enforces a
collection of risk constraints associated with a scalarization set,
further complicates the solution of these optimization models. A
common approach in solving such large-scale stochastic optimization
models is to employ a Benders-type scenario decomposition approach
\citep[e.g.,][]{Ruszczynski03}. The classical decomposition methods
cannot be directly applied to our model due to the complicating risk
constraints. Utilizing successive relaxations of the multivariate
polyhedral CVaR relation, we develop two types of exact and finitely
convergent delayed cut generation solution algorithms; the
iteratively generated cuts -- common in both algorithms -- are
associated with the scalarization vectors for which the risk
constraints are relaxed. In addition, the second algorithm adapts a
scenario decomposition approach that exploits the decomposable
structure of CVaR and the second-stage problems; this further
decomposition proves to be useful in solving larger problem
instances. We also develop strong duality results and optimality
conditions under certain linearity assumptions. Despite our focus on
the CVaR-based models, in Section \ref{sec:ssd}, we show that our
proposed solution methods also apply to the multivariate
SSD-constrained two-stage models. In Section
\ref{sec:disaster-model}, we apply the proposed modeling approach to
a stochastic pre-disaster relief network design problem and present
the corresponding mathematical programming formulations. Section
\ref{sec:comp-study} is dedicated to the computational study.

\section{Two-Stage Optimization with Multivariate CVaR
Constraints}\label{sec:general-model} In this section, we introduce
the multicriteria stochastic decision-making framework of interest,
and present the proposed multivariate CVaR-constrained two-stage
stochastic programming model and the corresponding delayed cut
generation solution algorithms.

We consider a finite probability space $(\Omega,2^\Omega,\Pi)$ with
$\Omega=\{\omega_1,\ldots,\omega_m\}$ and $\Pi(\omega_s)=p_s$. Let
$S=\{1,\ldots,m\}$ be the index set of the elementary events (also
referred to as scenarios). We assume that a benchmark random outcome
vector and a polyhedron of scalarization vectors, each component of
which represents the relative importance of each decision criterion,
are given (specified by decision makers). Let $\vZ$ be the benchmark
vector with realizations $\vz_1,\ldots,\vz_{|T|}$ and associated
probabilities $\tilde{p}_i,~i \in T$; it is often constructed from
some benchmark decision in which case we have $T=S$ and $\tilde
p_i=p_i,~i\in S$. The scalarization set $\C$ is naturally assumed to
be a subset of the unit simplex without loss of generality. In our
setting, smaller values of random variables and risk measures are
considered to be preferable. We now introduce the class of two-stage
stochastic programming problems with multivariate CVaR constraints,
where the general form of the first-stage problem is given by
\begin{subequations}\label{GeneralForm1}
\begin{align}
\min~~&f(\vx)+\E(Q(\vx,\vxi(\omega)))\label{GeneralForm1-obj}\\
\text{s.t.}~~& \cvar_{\alpha}(\vc^\top \vG(\vx,\vxi(\omega))) \leq \cvar_{\alpha}(\vc^\top \vZ), \quad \forall~\vc \in \C,\label{c:cvar}\\
&\vx \in \X.
\end{align}
\end{subequations}
Here, $f(\vx)$ is a convex objective function, $\X\subset \R_+^{\bar{n}_1}\times
\mathbb{Z}_+^{n_1-\bar{n}_1}$ is a non-empty convex set (such as a polyhedron) of the
first-stage decision variables, defined by the deterministic
constraints of the problem. $\vxi(\omega)\subset \R^r$ and
$Q(\vx,\vxi(\omega_s))$ designate the vector of the random input
parameters of the second-stage problem and the optimal second-stage
objective value under scenario $s$, respectively. More specifically,
introducing the notation
$\vxi(\omega_s)=(\vq(\omega_s),T(\omega_s),W(\omega_s),\\ \vh(\omega_s))=(\vq_s,T_s,W_s,\vh_s)$,
\begin{align}
Q(\vx,\vxi(\omega_s))=&\min\limits_{\vy \in \Y(\vx,\vxi(\omega_s))}
\quad \vq^\top_s  \vy,\label{second-stage-problem}
\end{align}where $\Y(\vx,\vxi(\omega_s)) = \{\vy \in \R_+^{n_2}:
T_s\vx+W_s\vy\ge \vh_s\}$. The second-stage objective function
$\vq_s^\top \vy$ could be a weighted sum of multiple objectives as
in \cite{DW16}, or a single objective, such as cost, that is
compatible with the first-stage objective. In addition, for
some mapping $\vG:\X\times \R^r\rightarrow\R^d$,
$\vG(\vx,\vxi(\omega))$ is the decision-based random vector featured
in \eqref{c:cvar}. In particular, $\vG(\vx,\vxi(\omega_s))$
represents the multiple random performance measures of interest
associated with the optimal second-stage decisions $\vy(\omega_s)
\in \argmin\limits_{\vy\in \Y(\vx,\vxi(\omega_s))}\{\vq^\top_s
\vy\},~s\in S,$ for given $\vx$ and $\vxi(\omega_s)$. To highlight
this correspondence we introduce the random vector
$\hat{\vG}(\vx,\vy)$ given by
$\left[\hat{\vG}(\vx,\vy)\right](\omega)=\hat{\vG}(\vx,\vy(\omega),\omega)=\vG(\vx,\vxi(\omega))$.
Finally, we let
$\hat{\vG}(\vx,\vy(\omega_s),\omega_s)=\hat{\vg}_s(\vx,\vy_s)=(\vg_s^1(\vx,\vy_s),\vg_s^2(\vx,\vy_s),\ldots,\vg_s^d(\vx,\vy_s))^\top,~s\in
S$, and assume that $\vg_s^i(\vx,\vy_s)$ is an affine function of
$\vx$ and $\vy_s$ such that $\vg_s^i(\vx,\vy_s) = \vbg_s^i \vx +
\vtg_s^i \vy_s$ for all $i\in \{1,\ldots,d\}$ and $s\in S$.
Throughout the paper, we assume that problem \eqref{GeneralForm1},
if feasible, has a finite objective value, i.e.,
$f(\vx),\E(Q(\vx,\vxi(\omega)))>-\infty$ for all $\vx \in \X.$

The multivariate risk constraint \eqref{c:cvar} ensures that the
decision-based random outcome vector $\vG(\vx,\vxi(\omega))$ is
preferable to the specified benchmark $\vZ$ according to the
multivariate CVaR relation introduced by \citet{Noyan13}. According
to this relation, all scalarized versions of $\vG(\vx,\vxi(\omega))$
are preferable to the corresponding scalarizations of the benchmark
outcome according to the univariate CVaR-based preference relation.
We note that using a scalarization-based risk constraint such as
\eqref{c:cvar} is useful to address ambiguities and inconsistencies
in the weight vectors; for detailed discussions we refer to
\cite{Mehrotra12} and \cite{Liu15}. Moreover, CVaR is a widely
applied popular risk measure with several desirable properties; it
is a law invariant coherent risk measure \citep{Artzner99}, serves
as a fundamental building block for other coherent risk measures,
and can be used to capture a wide range of risk preferences,
including risk-neutral and worst-case approaches. In addition,
\cite{Noyan13} highlight that CVaR-based relations provide
flexible, meaningful, and computationally tractable relaxations for
overly conservative SSD relations. In a broad sense, CVaR can be
viewed as a weighted sum of the least favorable outcomes (those
exceeding the $\alpha$-quantile -- also known as value-at-risk at
confidence level $\alpha$).

\ignore{
We next provide the
formal definition of CVaR along with an LP representation. For a
random variable $V$ with realizations $v_1,v_2,\ldots,v_m$ and
corresponding probabilities $p_1,p_2,\ldots,p_m$, the conditional
value-at-risk at confidence level $\alpha \in [0, 1)$ is defined as
\citep{Rockafellar00}:
\begin{subequations}
\label{equations}
\begin{align}
\cvar_{\alpha}(V)=&\min\left\{\eta + \dfrac{1}{1-\alpha} \E([V-\eta]_+)~:~ \eta \in \R\right\}\label{def:cvar1}\\
= &\min \bigg\{\eta + \dfrac{1}{1-\alpha}\sum_{s\in S} p_s
w_s~:~ w_s \ge v_s - \eta \nonumber \\
 &\hspace{4.5cm} \forall~ s\in S,~\vw\in \R_+^{m},\ \eta \in \R \bigg\},\label{def:cvar2}
\end{align}
\end{subequations}where $[z]_+ = \max(z, 0)$ denotes the positive part of a number $z\in
\R$. For risk-averse decision makers typical choices for $\alpha$
are small values such as $\alpha = 0.05$.}

Using the above notation, the proposed
risk-averse two-stage stochastic programming model can alternatively
be formulated as follows:
\begin{subequations}\label{GeneralForm2}
\begin{align}
\min~~&f(\vx)+\sum\limits_{s\in S}p_sQ(\vx,\vxi(\omega_s))\\
\text{s.t.}~~& \cvar_{\alpha}(\vc^\top \hat{\vG}(\vx,\vy)) \leq \cvar_{\alpha}(\vc^\top \vZ), \quad \forall~\vc \in \C,\label{form2:cvar}\\
&\vx \in \X,\\
& Q(\vx,\vxi(\omega_s))=\vq^\top_s  \vy(\omega_s), \quad \forall s\in S,\\
& \vy(\omega_s) \in \argmin\limits_{\vy\in
\Y(\vx,\vxi(\omega_s))}\{\vq^\top_s  \vy\}, \quad \forall s\in
S,\label{optcond:y}
\end{align}\end{subequations} where for a given $\hat \vc_{(l)}\in \C$,  $\cvar_{\alpha}(\hat \vc_{(l)}^\top \hat{\vG}(\vx,\vy))$
is calculated as \citep[see][]{Rockafellar00}
\begin{align}\label{eq:cvarg}
\min\quad
\bigg\{\eta_l + \dfrac{1}{1-\alpha}\sum_{s\in S}p_s w_{sl}~:~ & w_{sl}
\ge \hat \vc_{(l)}^\top \hat{\vG}(\vx,\vy(\omega_s),\omega_s) -
\eta_l \nonumber \\
& \forall~ s\in S,~w_{sl}\geq 0 \quad \forall~ s\in S,\
\eta_l \in \R \bigg\}.
\end{align}
For the risk-neutral version of the proposed two-stage optimization
model \eqref{GeneralForm1}-\eqref{second-stage-problem}, which does
not feature the risk constraints \eqref{c:cvar}, it is well-known
that the corresponding deterministic equivalent formulation (DEF)
can simply be obtained from \eqref{GeneralForm2} by dropping the
optimality condition on the second-stage decisions, i.e., by
replacing \eqref{optcond:y} with $\vy(\omega_s)\in
\Y(\vx,\vxi(\omega_s)),~\forall s\in S$. DEF provides both the
optimal first-stage and second-stage decisions without any issue of
the so-called second-stage sub-optimality; this result relies on the
interchangeability of the order of expectation and minimization
operators associated with the second-stage objective function, and
it is known as interchangeability principle \citep[see, e.g.,][]{Ruszczynski03}. 
Such a result also exists for the risk-averse
two-stage stochastic programming models, where the objective
function features a non-decreasing functional of the recourse cost
\citep{Takriti04}. The next proposition shows that we can similarly
drop the optimality condition in \eqref{optcond:y} to obtain the DEF
of our model even if it features a set of risk constraints.

\begin{proposition}\label{pro:DEFform}The following deterministic formulation is
equivalent to the two-stage model given by
\eqref{GeneralForm1}-\eqref{second-stage-problem} (and consequently,
equivalent to \eqref{GeneralForm2}):
\begin{subequations}
\label{def}
\begin{align}
\min~~&f(\vx)+\sum_{s \in S} p_s \vq^\top_s  \vy_s  \\
\text{s.t.}~~& \eta_l + \dfrac{1}{1-\alpha} \sum_{s \in S} p_s w_{sl} \leq \cvar_\alpha (\hat \vc_{(l)}^\top \vZ), \quad \forall~l=1,\dots,\bar L, \label{DEFForm-cvar-const1}\\
& w_{sl} \geq \hat \vc_{(l)}^\top \hat{\vg}_s(\vx,\vy_s) - \eta_l,\quad \forall~s \in S,~l=1,\dots,\bar L,\label{DEFForm-cvar-const2} \\
& w_{sl} \geq 0,\quad \forall~s\in S,~l=1,\dots,\bar L,\label{DEFForm-cvar-const3}\\
&\vx \in \X,\quad \veta \in \R^{\bar L},\label{DEFForm-cvar-const4}\\
&T_s\vx+W_s\vy_s \ge \vh_s,\quad \forall~s\in S,\label{eq:def-couple}\\
&\vy_s \in \R_+^{n_2},\quad \forall~s \in S,\label{eq:def-couple2}
\end{align}
\end{subequations}
where $\bar L$ is a finite integer and $\hat \vc_{(l)},~
l=1,\dots,\bar L,$ are the projections of the vertices of the
polyhedron $P= \{(\vc,\eta,\vw)\in \C\times\R\times \R^{|T|}_+: w_i
\geq \eta- \vc^\top \vz_i, \ i \in T\}$.
\end{proposition}

\begin{proof}
First, observe that it is sufficient to consider a finite number of
scalarization vectors from set $\C$, specifically the projections of
the extreme points (given by $\hat \vc_{(1)},\dots, \hat \vc_{ (\bar
L)}$) of the above defined polyhedron $P$ \cite[]{Noyan13}, since
$P$ is only characterized by the given fixed (decision-independent)
benchmark outcome vector. Hence, constraints \eqref{c:cvar} can be
replaced by
\begin{equation}
 \cvar_{\alpha}(\hat \vc_{(l)}^\top \vG(\vx,\vxi(\omega))) \leq \cvar_{\alpha}(\hat\vc_{(l)}^\top \vZ), \quad \forall~l=1,\dots,\bar L.\label{c:cvar-fin}
\end{equation}
Second, the univariate $\cvar$-relations in \eqref{c:cvar-fin} can
be represented by linear inequalities of type
\eqref{DEFForm-cvar-const1}-\eqref{DEFForm-cvar-const3}, ignoring
the issue of optimality of the $\vy_s,~s\in S,$ decisions for now.
This observation is a simple consequence of the LP representation
of CVaR in \eqref{eq:cvarg}, and the fact that $\cvar_{\alpha}(\vc^\top \vZ)$
is a known constant given a vector $\vc$ and the benchmark $\vZ$. It
is also easy to see that CVaR constraints
\eqref{DEFForm-cvar-const2} can be decomposed over scenarios. This
allows us to consider some modified versions of the first- and
second-stage problems given in
\eqref{GeneralForm1}-\eqref{second-stage-problem}; in particular,
$w_{sl}$ and $\eta_l$, for $s\in S, l=1,\dots,\bar L$ are considered
as first-stage decision variables and \eqref{c:cvar} is replaced by
\eqref{DEFForm-cvar-const1} and \eqref{DEFForm-cvar-const3}, whereas
the inequalities associated with $s\in S$ in
\eqref{DEFForm-cvar-const2} are added to the constraints of the
second-stage problem in \eqref{second-stage-problem}. Solving
\eqref{second-stage-problem} ensures that $\vy_s$ featured in
\eqref{DEFForm-cvar-const2} satisfies the optimality condition in
\eqref{optcond:y}. As a result, we obtain the following model which
is equivalent to  \eqref{GeneralForm1}-\eqref{second-stage-problem}:
\begin{subequations}\label{intermediate1}
\begin{align}
\min~~&f(\vx)+\sum\limits_{s\in S}p_s \tilde Q(\vx,\veta,\vw,\vxi(\omega_s))\\
\text{s.t.}~~
&\eqref{DEFForm-cvar-const1},\eqref{DEFForm-cvar-const3}-\eqref{DEFForm-cvar-const4},
\end{align}
\end{subequations}
\begin{subequations}\label{intermediate2}
where \begin{align}
\tilde Q(\vx,\veta,\vw,\vxi(\omega_s))=\quad  \min ~~&\vq^\top_s  \vy\\
\text{s.t.}~~& w_{sl} \geq \hat\vc_{(l)}^\top {\hat \vg}_s(\vx,\vy) - \eta_l,\quad \forall~l=1,\dots,\bar L, \\
&T_s\vx+W_s\vy \ge \vh_s,\quad \vy \in \R_+^{n_2}.
\end{align}
\ignore{ikisi ayni tip kisit degil, burada for all s\in S yok, o yuzden ben bu sekilde buradaki kisitlari silmemistim}
\end{subequations}
Thus, we can reformulate the proposed risk-averse
model \eqref{GeneralForm1}-\eqref{second-stage-problem} as a
risk-neutral one (as seen from
\eqref{intermediate1}-\eqref{intermediate2}). Then, by the
well-known interchangeability principle of risk-neutral two-stage
models, the assertion directly follows. 



\end{proof}

\begin{remark}\label{remark: mean-risk}\textbf{Mean-Risk Objective
Function.} The proposed model can be easily extended to the case
where the random objective outcomes are compared according to a
mean-risk functional representing the decision makers' risk
preferences. In particular, $\E(Q(\vx,\vxi(\omega))$ can be replaced
by $\E(Q(\vx,\vxi(\omega)) + \lambda \rho(Q(\vx,\vxi(\omega))$ in
\eqref{GeneralForm1}, where $\rho$ is a risk functional such as
$\cvar_\alpha$ and $\lambda$ is a non-negative trade-off coefficient
representing the exchange rate of mean cost for risk. When $\rho$ is
a non-decreasing function and $\lambda \geq 0$, the mean-risk
function preserves the convexity, and consequently, slight variants
of our cutting plane solution algorithms, which we will describe in
the next section, are applicable \citep[for similar developments
see, e.g.,][]{Ahmed06,Noyan12}. Thus, our proposed modeling and
solution framework is not limited to the risk-neutral objective
function.
\end{remark}

In the appendix we provide strong duality results and optimality
conditions for an important special case, \ignore{of the
multivariate CVaR-constrained two-stage model
\eqref{GeneralForm1}-\eqref{second-stage-problem}, where the mapping
$f$ is linear, the set $\X $ is polyhedral and the first-stage
decisions are continuous. These results} which generalize the
previous results for single-stage multivariate CVaR-constrained
problems \citep{Noyan13}. Related results for two-stage multivariate
SSD-constrained optimization are given in \cite{DW16}. Next we
present two types of  algorithms to solve the computationally
challenging two-stage stochastic programming model
\eqref{GeneralForm2}.

\subsection{Delayed Cut Generation for Deterministic
Equivalent Formulation} \label{dcg-def}

The  deterministic equivalent
formulation \eqref{def} features finite but potentially an exponential
number of scalarization vectors, obtained as projections
of the vertices of the polyhedron $P$. Hence, it is natural to
develop a delayed cut generation algorithm, which avoids the impractical
vertex enumeration approach required for explicitly constructing the
risk constraints in advance. The proposed algorithms rely on
solving successive relaxations of the multivariate polyhedral CVaR
relation, and they iteratively generate cuts associated with the
scalarization vectors for which the risk constraints are relaxed.

In our first cutting plane approach, at an intermediate iteration,
the relaxed master problem (RMP) includes the constraint
(\ref{form2:cvar}) for a subset of $\C$, say $\{\tilde
\vc_{(1)},\dots, \tilde \vc_{(L)}\}$. Accordingly, RMP takes the
form of \eqref{def}, where $\bar{L}$ is replaced by $L$.\ignore{The
RMP is given by
\begin{subequations}
\label{defmaster}
\begin{align}
\min~~&f(\vx)+\sum_{s \in S} p_s \vq^\top_s  \vy_s  \\
\text{s.t.}~~& \eta_l + \dfrac{1}{1-\alpha} \sum_{s \in S} p_s w_{sl} \leq \cvar_\alpha (\tilde{\vc}_{(l)}^\top \vZ), \quad \forall~l=1,\ldots,L, \label{defcut}\\
& w_{sl} \geq \tilde{\vc}_{(l)}^\top \hat{\vg}_s(\vx,\vy_s) - \eta_l,\quad \forall~s \in S,\ l=1,\ldots,L,\label{defcut2} \\
& w_{sl} \geq 0,\quad \forall~ s\in S,~ l=1,\ldots,L,\label{defcut3} \\
&\vx \in \X, \veta \in
\R^{L},\\
&T_s\vx+W_s\vy_s \ge \vh_s,\quad \forall~s\in S,\label{def:Y1}\\
&\vy_s \in \R_+^{n_2},\quad \forall~s \in S.\label{def:Y2}
\end{align}
\end{subequations}}
Given an optimal solution to the RMP $(\bar \vx,\bar \veta,\bar
\vw)$, the following separation problem (SP) is solved to identify
if there is any violated inequality \eqref{form2:cvar}:
\begin{align}
\max\limits_{\vc \in \C}\quad & \cvar_{\alpha}(\vc^\top
\hat{\vG}(\bar \vx,\bar \vy)) - \cvar_{\alpha}(\vc^\top \vZ).
\label{sep}
\end{align}
\citet{Noyan13}, \citet{Kucukyavuz14}, and \citet{Liu15} provide
alternative mixed-integer programming (MIP) formulations for the cut
generation problem (\ref{sep}). Because we can treat the DEF as a
large-scale single-stage multivariate CVaR-constrained problem, the
convergence of the resulting  delayed cut generation algorithm
follows from \cite{Noyan13}.

\begin{remark}\label{ref:sep-conv}
The cut generation formulations in \cite{Noyan13},
\cite{Kucukyavuz14} and \cite{Liu15} are based on the opposite
convention that larger values of random variables, as well as larger
values of risk measures, are considered to be preferable. Those
existing MIP formulations should be altered to reflect this
difference in convention.
\end{remark}

\subsection{Delayed Cut Generation with Scenario Decomposition}
\label{dcg-sd}
\ignore{
\comment{large-scale yapiya vurgu yapilip scenario dec
motive edilseydi, $y_s$ var deniliyor ama asil motivasyon not
edilmiyor sanki? scetion direk ``Note that'' seklinde
baslamasa?Senaryoya bagli recourse-dec avoid edilmesi, bunun
two-stage modellerde common yaklasim oldugu, etc...}
}

The deterministic equivalent formulation \eqref{def} contains a
vector of $\vy_s$ variables for each scenario $s\in S$ and its size
grows very fast as the number of scenarios increases. Because the
number of scenarios is typically large, it is generally impractical
to solve the DEF directly, even without the multivariate stochastic
benchmarking constraints. The L-shaped method \citep{VW69}, which is
a Benders-type scenario decomposition algorithm, is arguably the
most commonly used computationally viable method to solve the
classical (risk-neutral) two-stage stochastic programs. We propose a non-trivial
extension of this approach to
solve our multivariate CVaR-constrained two-stage model.

\ignore{
\todo{


In the DEF of a classical (risk-neutral) two-stage stochastic
program, once the first-stage variables are fixed,  the second-stage
problems decompose by scenarios. With this observation, an
adaptation of the Benders decomposition method can be employed to
decompose the problem by scenarios. This method, known as the
L-shaped method \citep{VW69}, is arguably the most commonly used
computationally viable method to solve the large-scale two-stage
stochastic optimization problems.  This method solves a so-called
Benders master problem that contains only the first-stage variables
and auxiliary variables for the second-stage value functions.}}

In the DEF \eqref{def}, constraints
\eqref{DEFForm-cvar-const1}--\eqref{DEFForm-cvar-const2}, which
capture the CVaR relation \eqref{c:cvar} along with the associated
auxiliary variables $\veta$ and $\vw$, create further coupling of
the scenarios in addition to the original coupling constraint
\eqref{eq:def-couple}. Considering this structure, we handle the
$\vx$, $\veta$, and $\vw$ variables in the first stage, decompose
the second-stage problems over scenarios once the first-stage
variables are fixed, and solve iteratively the RMP involving only
the first-stage decision variables and auxiliary decision variables
($\theta_s,~s\in S$) for approximating the optimal second-stage
objective function values. We obtain the following RMP at an
intermediate iteration, where a subset of the scalarization vectors
of cardinality $L$ is generated so far:
\begin{subequations}
\label{master}
\begin{align}
\hspace{-0.5cm}\text{(MP$^L$)}\ \min \quad & f(\vx) + \sum_{s\in S} p_s \theta_s\\
\text{s.t.}\quad & ( \vx, \veta,
\vw, \vtheta) \in \mathcal O,  \label{optcut-cl}\\
& ( \vx, \veta,
\vw) \in \mathcal F ,\label{feascut-cl}\\
& \eta_l + \dfrac{1}{1-\alpha} \sum_{s \in S} p_s w_{sl} \leq \cvar_\alpha (\tilde{\vc}_{(l)}^\top \vZ), \quad \forall~l=1,\ldots,L,\label{cvarcut}\\
&\eqref{DEFForm-cvar-const3}-\eqref{DEFForm-cvar-const4} \text{
(with } \bar{L} \text{ is replaced by } L).
\end{align}
\end{subequations}
Here $\mathcal O$ is a polyhedral set defined by the constraints
(referred to as the {\it optimality cuts}) that give valid lower
bounding approximations of $\vtheta$, and $\mathcal F$ is a
polyhedral set defined by the constraints (referred to as the {\it
feasibility cuts}) that represent the conditions for the first-stage
decision vectors to yield feasible second-stage problems. In what
follows, we give the explicit forms of the inequalities in the sets
$\mathcal O$ and $\mathcal F$.

We start the algorithm with a small subset (could be empty) of the
scalarization set $\C$. At each iteration, given the subset
$\{\tilde \vc_{(1)},\dots ,\tilde \vc_{(L)}\}$, we first solve the
relaxed master problem $(\MP^L)$ to obtain an optimal solution $\bar
\vv:=(\bar \vx,\bar \veta,\bar \vw,\bar \vtheta)$. Using this
solution of the RMP, we solve the second-stage subproblem
\text{(PS$_{s}^L$)} for each scenario $s\in S$ given by
\begin{subequations}
\label{primSub}
\begin{align}
\text{(PS$_{s}^L$)}\quad \bar Q_s^L:=\min \quad & \vq^\top_s  \vy\\
\text{s.t.}\quad & -\tilde{\vc}_{(l)}^\top \vtg_s \vy \ge \tilde{\vc}_{(l)}^\top \vbg_s \bar \vx -\bar \eta_l - \bar w_{sl}, \quad \forall~l=1,\ldots,L, \label{C:p1} \\
&W_s\vy \ge \vh_s-T_s \bar \vx ,\label{C:p2}\\\ &\vy \in
\R_+^{n_2}\label{C:p3},
\end{align}
\end{subequations}
where constraint \eqref{C:p1} is equivalent to constraint \eqref{DEFForm-cvar-const2}
 after partial substitution of the values of the first-stage variables  with $\bar \vx, \bar \veta$ and $\bar \vw$.
\ignore{ The  dual of problem \eqref{primSub} is
\begin{subequations}
\label{dualSub}
\begin{align}
\text{(DS$_{s}^L$)}\quad \max \quad & \vgamma(\vh_s-T_s \bar \vx) + \sum_{l=1}^L \beta_{l}\left(\tilde{\vc}_{(l)}^\top \vbg_s \bar \vx -\bar \eta_l - \bar w_{sl} \right)\\
\text{s.t.}\quad & W_s^\top \vgamma - \sum_{l=1}^L \vtg_s^\top \tilde{\vc}_{(l)} \beta_{l} \leq \vq_s,\\
& \vgamma,\ \vbeta \geq 0,
\end{align}
\end{subequations}
}Let $\beta_{sl},~l=1,\ldots,L,$ and $\vgamma_s$ for $s\in S$, be
the dual variables associated with constraints \eqref{C:p1} and
\eqref{C:p2}, respectively. If the primal subproblem
\text{(PS$_{s}^L$)} is infeasible for some $s \in S$, then let $(\vgamma_s, \vbeta_s)$ provide an extreme ray (direction of
unboundedness) of the corresponding dual feasible region and add the
following Benders feasibility cut to the set $\mathcal F$:
\begin{align}
 0 \geq  \vgamma_s^\top(\vh_s-T_s\vx) +  \sum_{l=1}^L \beta_{sl}^\top \left(\tilde{\vc}_{(l)}^\top \vbg_s \vx -\eta_l - w_{sl}\right).\label{feascut}
\end{align}
If the primal subproblem is feasible and the current estimate of the
optimal second-stage objective value ($\bar \theta_s$) is less than
the actual optimal second-stage objective value ($\bar Q_s^L$), then
let $(\vgamma_s, \vbeta_s)$ provide an optimal dual
extreme point and add the following Benders optimality cut to the
set $\mathcal O$:
\begin{align}
 \theta_s \geq  \vgamma_s^\top(\vh_s-T_s\vx) + \sum_{l=1}^L \beta_{sl}^\top \left(\tilde{\vc}_{(l)}^\top \vbg_s \vx -\eta_l - w_{sl} \right). \label{optcut}
\end{align}

Now suppose that all primal subproblems are feasible with the finite
objective values $\bar Q_s^L, s\in S$ (recall our assumption that
the original problem is not unbounded). In this case, given $(\bar
\vx,\bar \vy)$, we solve the separation problem \eqref{sep} to
obtain an optimal solution  $\vc^*$. If the corresponding optimal
objective value is non-positive, then there is no violation in the
CVaR constraint \eqref{form2:cvar}. Furthermore, recall that in this
case, all second-stage problems are feasible as well. In other
words, the current first-stage solution is feasible for the original
problem, so we update the upper bound on the optimal objective value
(denoted by $u$). On the other hand, if the optimal objective value
of the SP is positive, then $\vc^* \in \C$ creating the most
violation in constraint \eqref{C:p1} is added as the $L+1$th
scalarization vector both to the RMP and the second-stage
subproblems with $L \gets L+1$. We repeat these iterations until
either an infeasible RMP is detected, or the RMP objective function
value is within the given tolerance $\epsilon$ of the upper bound
$u$. The optimality tolerance is specified as a very small positive
number such as $10^{-10}$. The pseudo-code of the proposed
decomposition algorithm is provided in Algorithm \ref{alg}. Note
that the feasible regions of the dual subproblems change depending
on the scalarization vectors included in the formulations at that
iteration. In addition, the RMP also grows both in terms of the
number of variables and constraints as we add more scalarization
vectors to its formulation due to the addition of the $\eta_l$ and
$\vw_l$ variables for each scalarization vector $\tilde\vc_{(l)}$.
As a result, our proposed algorithm can be seen as a \emph{delayed
column and cut generation algorithm}.

\begin{algorithm}[ht]
Given $L$ initial scalarization vectors ($\{\tilde
\vc_{(i)}\}_1^L\subset \mathcal C$) and a small tolerance parameter
$\epsilon>0$, set $converge=false$ and $u \gets \infty$\;
\While{$converge=false$}{Solve $(\MP^L)$ and get the optimal
solution $\bar \vv$ and objective value $F(\MP^L)$\; \eIf{$(\MP^L)$
is infeasible}{$converge=true$. The original risk-averse problem is
infeasible\;} {\eIf{$(1-\epsilon)u\leq F(\MP^L)\leq (1+\epsilon)u$}{
Set $converge=true$. The optimal solution is $\bar
\vv$\;\label{alg:goto} } {
 \For{each scenario $s \in S$}{
Solve the primal subproblem \text{(PS$_{s}^L$)} and get the optimal
objective value $\bar Q_s^L$\; \eIf{(\text{PS$_{s}^L$}) is
infeasible} {Obtain an extreme ray of the dual feasible region
associated with (\text{PS$_{s}^L$}), given by $(\vgamma_s,
\vbeta_s)$, and add the corresponding feasibility cut to $\mathcal
F$\;} {Obtain an optimal extreme point of the dual of
\text{(PS$_{s}^L$)}, given by $(\vgamma_s, \vbeta_s)$, and if $\bar
\theta_s<Q_s^L$, then add the corresponding optimality cut to
$\mathcal O$\;}} \If{\text{(PS$_{s}^L$)} is feasible for all $s \in
S$}{ Solve the separation problem (SP)\; Obtain the optimal solution
$\vc^*$ and optimal objective value $v(\vc^*)$\; \eIf{$v(\vc^*)\leq
0$}{ Current solution is feasible, set $u \gets f(\bar
\vx)+\sum_{s\in S}p_s \bar Q_s^L $\; }{ Set $L \gets L+1$ and
$\tilde{\vc}_{(L)} \gets \vc^*$\; } } }} } \caption{Decomposition
algorithm for the multivariate risk-constrained two-stage
optimization \label{alg}}
\end{algorithm}

\begin{proposition} \label{prop:cvar-finite}
Suppose that the relaxed master problem $(\MP^L)$ is given by
\eqref{master}, the second-stage subproblems \text{(PS$_{s}^L$)} are
given by \eqref{primSub}, the separation problem is given by
\eqref{sep}, the feasibility and optimality cuts, respectively, are
given by \eqref{feascut} and \eqref{optcut}, and the first-stage
solution vector is given by $\bar \vv=(\bar \vx,\bar \veta,\bar \vw,
\bar \vtheta)$. Then, Algorithm \ref{alg} provides either an optimal
solution to problem \eqref{GeneralForm1} or a proof of infeasibility
in finitely many iterations.
\end{proposition}

\begin{proof}
First, observe that the optimality cuts \eqref{optcut} generated
from subproblems \text{(PS$_{s}^L$)} for $s\in S$ and the subset
$\{\tilde \vc_{(1)},\dots, \tilde \vc_{(L)}\}$ of $\C$ are valid,
because subproblems \text{(PS$_{s}^L$)} are relaxations of the
original subproblems \eqref{intermediate2} given by
(\text{PS$_{s}^{\bar L}$}), and hence the duals of the relaxed
subproblems provide valid lower bounds on the optimal second-stage
objective values. Similarly, the feasibility cuts \eqref{feascut}
obtained from the relaxed subproblems are valid for the original
subproblems \eqref{intermediate2}. At an intermediate iteration, if
there is a violation in the CVaR constraint \eqref{form2:cvar}, then
the exact solution of the SP guarantees that a corresponding
violating scalarization vector is identified and the associated
inequalities and variables are added to the RMP. Furthermore,
\cite{Noyan13} provide a procedure to guarantee that the $\vc^*$
vector generated from the SP is a vertex solution as defined in
Proposition \ref{pro:DEFform}. Next, recall that the number of
scalarization vectors of interest, $\bar L$, is finite. Hence, in
the worst case, in a finite number of iterations, all $\bar L$
scalarization vectors will be added, and the second-stage problems
become exact. At this point, the convergence of the algorithm
follows directly from the convergence of the classical L-shaped
method \cite[]{VW69}. 
\end{proof}

\section{Two-Stage Optimization with a Multivariate Stochastic
Ordering} \label{sec:ssd} The solution algorithms proposed in
Section \ref{sec:general-model} can also be applied to the two-stage
stochastic programs with the multivariate preference constraints
based on increasing convex order (ICO). The counterpart of ICO in
the opposite convention, where larger values of random variables are
preferred, is referred to as the SSD. Similar to the CVaR case, the
univariate ICO relation can be extended to the multivariate case by
considering a family of linear scalarization functions; more
specifically, a random vector $\vX$ is said to dominate another
random vector $\vZ$ with respect to the ICO and the scalarization
set $\C$ if
\begin{equation}
\E\left(\left[\vc^\top \vX - \eta \right]_+\right) \leq
\E\left(\left[\vc^\top \vZ - \eta \right]_+\right), \quad
\forall~\eta \in \R,\ \vc \in \C. \label{cons:ico1}
\end{equation}
The risk-averse two-stage optimization model of \citet{DW16}
utilizes this multivariate stochastic preference relation to compare
the decision-based random outcome vector $\hat{\vG}(\vx,\vy)$ with
the benchmark outcome vector $\vZ$. In particular, the two-stage
stochastic programming model with the multivariate ICO constraints
takes the form of \eqref{GeneralForm2}, where the multivariate CVaR
relation \eqref{form2:cvar} is replaced by \eqref{cons:ico1} with
$\vX=\hat{\vG}(\vx,\vy)$. Assuming again that $\vZ$ has finitely
many realizations $\vz_1,\ldots,\vz_{|T|}$ with associated
probabilities $\tilde{p}_i,~i \in T$, the inequality
(\ref{cons:ico1}) can be equivalently stated as
\begin{equation}
\E\left(\left[ \vc^\top \vX - \vc^\top \vz_t \right]_+\right) \leq
\E\left(\left[\vc^\top \vZ - \vc^\top \vz_t \right]_+\right), \quad
\forall~t \in T,\ \vc \in \C.\label{c:ssd}
\end{equation}
Relation (\ref{c:ssd}) directly follows from \eqref{cons:ico1} due
to the well-known result \citep[][]{Dentcheva03} that for finite
probability spaces the univariate SSD relation remains valid if the
corresponding expected shortfall constraints (in our convention,
constraints (\ref{cons:ico1}) featuring the expected excess values
for a particular $\vc$) only required for the realizations of the
benchmark random variable (instead of all $\eta \in \R$).

For finite probability spaces, \citet{Mello09} show that it is
sufficient to consider a finite number of scalarization vectors
$\{\hat \vc_{(1)},\dots,\hat \vc_{(\bar L)}\}$ in relation
\eqref{c:ssd}, specifically the projections of the extreme points of
polyhedra $P_t,~t\in T$, defined for each $t \in T$ as $P_t = \{w_i
\geq \vc^\top \vz_i - \vc^\top \vz_t,\ i \in T,\ \vc \in \C,\ \vw
\in \R^{|T|}_+ \}$. For an alternative finite representation of
\eqref{c:ssd}, which is based on the vertices of the polyhedron $P$
defined in Proposition \ref{pro:DEFform}, we refer to
\cite{Noyan13}. The fact that these finite representations of
\eqref{c:ssd} are only characterized by the given benchmark vector
allows us to develop computationally tractable DEFs of the
optimization models featuring the multivariate ICO relation as a
benchmarking constraint. In particular, for the two-stage
multivariate ICO-constrained problem we can obtain the following DEF
with finitely many constraints:
\begin{subequations}
\label{def-ssd}
\begin{align}
\min~~&f(\vx)+\sum_{s \in S} p_s \vq^\top_s  \vy_s  \\
\text{s.t.}~~&  \sum_{s \in S} p_s w_{slt} \leq \sum_{i \in T} \tilde{p}_i \left[\hat{\vc}_{(l)}^\top \vz_i - \hat{\vc}_{(l)}^\top \vz_t \right]_+, \quad \forall~t \in T,\ l=1,\ldots,\bar L, \label{ssd:defcut}\\
& w_{slt} \geq \hat{\vc}_{(l)}^\top \hat \vg_s(\vx,\vy_s) - \hat{\vc}_{(l)}^\top \vz_t,\quad \forall~s \in S,\ t\in T,\ l=1,\ldots,\bar L,\label{ssd:defcut2} \\
& w_{slt} \geq 0,\quad \forall~s\in S,\ t \in T,\ l=1,\ldots,\bar L, \label{ssd:defcut3}\\
&\vx \in \X,\quad \eqref{eq:def-couple}-\eqref{eq:def-couple2}.
\end{align}
\end{subequations}

Although the number of inequalities
\eqref{ssd:defcut}--\eqref{ssd:defcut2} is finite, it is possibly
exponential. Hence, as in the $\cvar$-constrained optimization, the
ICO constraints can be added as needed using a delayed cut
generation approach. As in Section \ref{dcg-def}, we start with
solving the DEF with a subset of constraints
\eqref{ssd:defcut}--\eqref{ssd:defcut2} for $l=1,\ldots, L$. For a
given solution $\bar \vv:=(\bar \vx,\bar \vw,\bar \vy)$ to this
relaxed problem, we check if there is a violation in \eqref{c:ssd}
(with $\vX=\hat{\vG}(\bar \vx,\bar \vy)$). Note that constraints
(\ref{c:ssd}) are defined for each scalarization vector and for each
realization of $\vZ$, as such possible violations in constraints
\eqref{c:ssd} should be checked for all $\vz_t,~t \in T$. In line
with these discussions, given $(\bar{\vx},\bar \vw,\bar{\vy})$, the
separation problem corresponding to the $t^{th}$ realization of
$\vZ$ becomes
\begin{align}
\min\limits_{\vc \in \C} \quad & \E\left(\left[\vc^\top \vZ -
\vc^\top \vz_t \right]_+\right) - \E\left(\left[ \vc^\top
\hat{\vG}(\bar{\vx},\bar{\vy}) - \vc^\top \vz_t
\right]_+\right).\label{c:ssd-sep}
\end{align}
This cut generation problem can be solved using the adaptations of the MIP formulations
provided in \cite{Mello09}, \cite{Kucukyavuz14}, and
\cite{Noyan16} for the opposite convention that larger outcomes are preferred.

As in the $\cvar$ case, the structure of the ICO constraints can be
exploited to further decompose the formulation \eqref{def-ssd} over
scenarios. Observe that constraints \eqref{ssd:defcut} and
\eqref{eq:def-couple} couple the scenarios together. However, if we
handle the original first-stage variables $\vx$ and the auxiliary
$\vw$ variables in the first stage, then the second-stage problems
decompose over scenarios once the first-stage variables are fixed.
Therefore, we can apply Algorithm \ref{alg} with the updated
definitions of the problems $(\MP^L)$, \text{(PS$_{s}^L$)}, and
(SP), the feasibility and optimality cuts, and the first-stage
solution vector $\vv$, which we describe next.

The RMP at an intermediate iteration, where a subset of the
scalarization vectors of cardinality $L$ is generated so far, is
formulated as
\begin{subequations}
\label{master-ssd}
\begin{align}
\hspace{-0.5cm}\text{(MP$^L$)}\ \min \quad & f(\vx) + \sum_{s\in S} p_s \theta_s\\
\text{s.t.}\quad & ( \vx,
\vw, \vtheta) \in \mathcal O,   \label{ssd:optcut-cl}\\
& ( \vx,
\vw) \in \mathcal F ,\label{ssd:feascut-cl}\\
& \vx \in \X, \quad \eqref{ssd:defcut},\eqref{ssd:defcut3}\text{
(with } \bar{L} \text{ is replaced by } L).
\end{align}
\end{subequations}
Given a first-stage solution vector $\bar \vv=(\bar \vx,\bar \vw,
\bar \vtheta)$ of RMP and the $L$ scalarization vectors, the
second-stage subproblem under $s\in S$ takes the form of
\begin{subequations}
\label{primSub-ssd}
\begin{align}
\text{(PS$_{s}^L$)}\quad \bar Q_s^L:= \min \quad & \vq^\top_s  \vy\\
\text{s.t.}\quad & -\tilde{\vc}_{(l)}^\top \vtg_s \vy \ge \tilde{\vc}_{(l)}^\top \vbg_s \bar \vx - \tilde{\vc}_{(l)}^\top \vz_t - \bar w_{slt}, \quad \forall~t \in T,\ l=1,\ldots,L, \label{C:p1-ssd} \\
&\eqref{C:p2}-\eqref{C:p3}.
\end{align}
\end{subequations}
Here, constraint \eqref{C:p1-ssd} is equivalent to constraint
\eqref{ssd:defcut2} with partial substitution of the values of the
first-stage variables $\bar \vx$ and $\bar \vw$. Denoting the dual
vectors associated with constraints \eqref{C:p1-ssd} and
\eqref{C:p2} by $\beta_{stl}, t\in T, l=1,\dots, L,$ and
$\vgamma_s,~s\in S,$ the feasibility and optimality cuts are given
by
\begin{align} 0 \geq  \vgamma_s(\vh_s-T_s\vx) +  \sum_{t\in
T}\sum_{l=1}^L \beta_{stl}\left(\tilde{\vc}_{(l)}^\top \vbg_s \vx
-\tilde{\vc}_{(l)}^\top \vz_t - w_{slt}\right) \text{ and
}\label{feascut-ssd}\\
 \theta_s \geq  \vgamma_s(\vh_s-T_s\vx) +  \sum_{t\in T}\sum_{l=1}^L \beta_{stl}\left(\tilde{\vc}_{(l)}^\top \vbg_s \vx -\tilde{\vc}_{(l)}^\top \vz_t - w_{slt} \right). \label{optcut-ssd}
 \end{align}

\begin{proposition}
Suppose that the relaxed master problem $(\MP^L)$ is given by
\eqref{master-ssd}, the second-stage subproblems
\text{(PS$_{s}^L$)} are given by \eqref{primSub-ssd}, the separation
problem is given by \eqref{c:ssd-sep}, the feasibility and
optimality cuts, respectively, are given by \eqref{feascut-ssd} and
\eqref{optcut-ssd}, and the first-stage solution vector is given by
$\bar \vv=(\bar \vx,\bar \vw, \bar \vtheta)$. Then, Algorithm
\ref{alg} provides either an optimal solution to problem
\eqref{def-ssd} or a proof of infeasibility in finitely many
iterations.
\end{proposition}

\begin{proof}
The proof is similar to that of Proposition \ref{prop:cvar-finite}. 
\end{proof}

Finally, we highlight four major advantages of our proposed solution
framework over the existing methods \citep{DW16}. First, our method
reformulates the risk-averse model of interest as a risk-neutral
two-stage stochastic program without using the Lagrangian relaxation
of the multivariate ICO constraints, and consequently, manages to
solve the problem within a single Benders decomposition framework.
It maintains a single master problem to solve successively redefined
(based on the iteratively enlarged subset of the scalarization
vectors) risk-neutral relaxations of the problem. Recall that using
a subset of $\C$ provides a relaxation of the problem. These
risk-neutral formulations are constructed by using the decomposable
structure of the ICO constraints over scenarios. \citet{DW16} also
construct successive relaxations by considering a subset of $\C$;
however, they obtain approximate risk-neutral relaxations by using
the Lagrangian relaxation of the multivariate ICO constraints. In
particular, at each iteration of their algorithms, a risk-neutral
approximation with an updated Lagrangian objective function is
solved (to calculate the dual function given a set of Lagrangian
multipliers) using decomposition methods such as Benders
decomposition. The need for solving a separate risk-neutral
two-stage model at each iteration, in addition to the well-known
computational challenges in solving the non-differentiable
Lagrangian dual problem, could impose significant computational
difficulties. Second, the Lagrangian-based existing algorithms are
not exact even for the finite probability spaces; they are shown to
finitely converge to an $\varepsilon$-feasible (with respect to the
ICO constraint) $\varepsilon$-optimal solution both for continuous
and finite probability spaces. In contrast, our finitely-convergent
algorithm provides an exact (optimal and feasible) solution for
finite probability spaces. Third, the existing algorithms do not
provide valid upper bounds on the optimal objective value at
intermediate stages, because the Lagrangian dual at an iteration is
constructed using the information on a subset of  $\C$. In contrast,
Algorithm \ref{alg} provides valid upper bounds (in line 20) at
intermediate stages. Fourth, our method applies to problems that
contain discrete variables in the first stage. In contrast, the
Lagrangian-based methods cannot guarantee the integer feasibility
(see the primal solution recovery step of the algorithms in
\citealp{DW16}). In summary, we contribute to the literature by
providing a new computationally tractable and exact solution
algorithm for the multivariate ICO-constrained two-stage models with
integer variables in the first stage under finite probability
spaces.

\section{A Stochastic Optimization Model for Pre-disaster Relief Network Design}\label{sec:disaster-model}
An effective and sound pre-disaster relief network design calls for
modeling the risk associated with the high level of uncertainty
inherent in rarely occurring disaster events \cite[see, e.g.,][]{Noyan12} and considering multiple and possibly conflicting
performance criteria such as responsiveness and equity
\cite[see, e.g.,][]{Vitoriano11,Huang2012,Gutjahr16}. Moreover,  a
two-stage decision making framework is beneficial in this context,
since the pre-positioning design decisions must be made before a
disaster strikes, whereas the relief distribution decisions should
be made in the post-disaster stage. Motivated by the significance of
developing such effective pre-disaster plans \cite[see, e.g.,][]{Balcik08,Apte2010}, we apply our proposed approach to a
stochastic pre-disaster relief network design problem. In this
section, we first briefly review the relevant literature, and then
describe the problem setting and present the corresponding
mathematical programming formulations.

\ignore{In particular, we introduce a new risk-averse optimization
model, which mainly determines the locations of the response
facilities along with their inventory levels before a disaster
strikes...Raws makalesini tartisirken bu kararlar anlatiliyor, sonra
we develop a risk-averse variant of it deyince bizim kararlar da
belli, zaten sonra model kisminda tekrar acikliyoruz}

There is a rich and growing literature on developing stochastic
programming models for humanitarian logistics \cite[e.g.,][]{Celik2012,Liberatore2013}. The majority of the existing
studies focus on pre-disaster relief network design and develop
risk-neutral two-stage stochastic programs \cite[e.g.,][]{Balcik08,Rawls2010,Apte2010,Doyen2011}. However, this
extensive literature includes only a few studies \cite[e.g.,][]{Rawls2011,Noyan12,Hong2015,OzgunThesis16} that provide
risk-averse stochastic models, and lacks models with multivariate
risk constraints. Here, we mainly discuss the studies that are most
closely related to ours. \citet{Rawls2010} develop a risk-neutral
two-stage stochastic programming model; in the first-stage it
determines the cost-wise optimal decisions on locations and
capacities of the response facilities, and the inventory levels of
relief supplies under uncertainty in demand, damage levels of
pre-stocked supplies, and transportation network conditions. Their
second-stage problem is formulated as a classical network flow
model, which involves detailed distribution decisions representing
the flow of relief supplies on each arc. \citet{Noyan12} obtains a
risk-averse version of this model by incorporating CVaR as the risk
measure on the total cost in addition to its expectation. There also
exist chance-constrained variants \citep{Rawls2011,Hong2015}. A
recent study \cite{OzgunThesis16} proposes a more elaborate
risk-averse extension, which features a mean-risk objective function
on the random total cost with CVaR being the risk measure (as in
\citealp{Noyan12}), and enforces a joint chance constraint on the
feasibility of the second-stage problem (as in \citealp{Hong2015}). In
contrast to the other variants, the model of \citet{OzgunThesis16}
relies on an alternative formulation of the second-stage problem,
which focuses on assigning the demand points to the facilities
instead of determining the detailed arc-flow decisions. In our
study, we follow this practically meaningful assignment-based
modeling approach for the second-stage problem, and develop a new
risk-averse variant of the widely-cited model of \citet{Rawls2010}.
\ignore{{Rawls2010} tekrarlayarak bitirdim ki bunu soyledikten sonra
diger section da buradan devam ediyor, bence boyle iyi...bu vurgu
introductiondan da silinmisti}

\ignore{This assignment-based formulation relies on a restricted
version of the classical flow conservation with the objective of
providing a more structured flow considering the coverage
issues...kisitlar aciklanirken gececek}

\ignore{In summary, we contribute to the humanitarian relief
literature by introducing a novel risk-averse optimization model,
which provides a flexible and tractable way of considering the
decision makers' risk preferences based on multiple stochastic
criteria.}

\ignore{directly bu modeli extend ettigimizi soylemek yukaridan
sonra natural bence...boylece direk additional kritere
atlayabiliriz..CVaR ile olan aciklama da onemli ve gerekli, comp
study'de buna odaklanildigindan, bu aciklamadan sonra bir daha comp
study'de aciklama vermeye gerek yok}

\vspace{0.5cm}

\noindent{\bf Problem Description and Mathematical Formulations.} We
extend the model of \citet{Rawls2011} by incorporating the
multivariate CVaR relation into the first-stage to evaluate the
decisions based on additional multiple stochastic criteria. In
particular, we consider the following additional stochastic
measures: the maximum proportion of unsatisfied demand and a
responsiveness measure based on the total delivery amount-based
average travel time. With the term responsiveness we refer to the
expeditiousness of the response. A large variety of criteria have
been employed in humanitarian logistics and they can be grouped in
three main categories \cite[see, e.g.,][]{Huang2012,Gutjahr16}:
\textit{efficiency} (related to cost), \textit{efficacy or
effectiveness} (related to providing quick and sufficient
distribution) and \textit{equity} (related to providing a fair
service). According to this classification, our model addresses the
issues about efficiency and efficacy of the relief operations using
a weighted-sum based objective, which aggregates the costs of
opening facilities, demand shortages, and purchasing and shipping
the relief supplies. In addition, it addresses the issues about
responsiveness and equity (in terms of supply allocation) via the
multivariate CVaR constraints. The
CVaR-based relation is preferred, since it is a natural relaxation
of its SSD-counterpart. Multivariate SSD constraints are typically very demanding
and they can be excessively hard to satisfy in practice.

\ignore{can be measured based on the response times, travel
distances, demand satisfaction levels, etc... in terms of supply
allocation and response times}

We model the relief distribution system using a network, where $I$
and $J$ denote the sets of nodes representing the demand and
candidate facility locations, respectively; we assume without loss
of generality that $J \subseteq I$. The set $L$ denotes the multiple
facility types, a facility of type $l$ has a given capacity level
$K_l$. Setting up a facility of type $l$ at node $j$ has a fixed
cost $F_{jl}$ and a unit variable acquisition cost $a_j$. We
consider a single commodity \cite[as in, e.g.,][]{Hong2015}, since
the relief items can be supplied as bundles. The demand values,
travel times, demand shortage penalty costs, and damage levels of
supplies are assumed to be random and their realizations are
represented by a finite set of scenarios $S$ with probabilities
$p_s,~s\in S$. We next introduce notation for the realizations of
these random parameters under scenario $s$: demand at node $i$ is
$d_i^{s}$, travel time from node $j$ to node $i$ is $\nu_{ji}^s$,
unit shortage penalty cost is $\pi^s$, unit shipment cost from node
$j$ to node $i$ is $c_{ji}^{s}$, proportion of undamaged supplies at
node $j$ is $\gamma_j^{s}$. In contrast to \cite{Rawls2010}, each
demand location can be served only by a facility node, which
satisfies a responsiveness requirement. In particular, we consider a
common upper bound on travel times to construct the
scenario-dependent coverage sets: $N^s_i$ and $M^s_j$, respectively,
denote the set of facility nodes that can cover demand node $i$ and
the set of demand nodes that can be served by facility node $j$
under scenario $s$. Enforcing a common threshold serves the
objective of providing equitable service in terms of response times.

\ignore{In our two-stage decision making framework, the first-stage
problem determines the sizes and locations of the response
facilities and the amount of relief supplies pre-positioned at each
facility, whereas the second-stage problem focuses on decisions on
the assignments of the demand points to facilities and the delivery
amounts to the demand points under each scenario.}

In our two-stage decision making framework, the first-stage
decisions are represented using the following notation: $x_{jl}= 1$
if a facility of size category $l \in L$ is located at node $j \in
J$, and $x_{jl}= 0$ otherwise; $R_{j}$ denotes the amount of
supplies pre-located at node $j \in J$. The notation for the
second-stage decisions under scenario $s\in S$ is as follows:
$y^s_{ji}$ denotes the amount of supplies delivered to demand node
$i \in I$ from node $j \in J$; $u^s_{i}$ designates the amount of
the unsatisfied demand at node $i \in I$; $u^s_{\max}$ is the
maximum proportion of unsatisfied demand (maximum is calculated over
the demand nodes), i.e., $u^s_{\max}=\max_{i \in I}u^s_{i}/d_{i}^s$.
We next present our risk-averse two-stage stochastic programming
model, where the first-stage problem is formulated as follows:
\begin{subequations}
\label{app:first}
\begin{align}
\min \quad & \sum\limits_{j \in J}\sum\limits_{l \in
L}F_{jl}x_{jl}+\sum\limits_{j \in J}a_jR_j+\sum\limits_{s\in
S}p_sQ(\vx,\vR,\vxi(\omega_s))\label{1stage-obj}
\\
 \text{s.t.}\quad
&\cvar_{\alpha}(\vc^\top \vG(\vx,\vR,\vxi(\omega))) \leq \cvar_{\alpha}(\vc^\top \vZ), \quad \forall~\vc \in \C, \label{con:cvarapp}\\
& \sum\limits_{l \in L}x_{jl}\leq 1,\quad \forall~j \in J,\label{numfacility}\\
&R_j\leq \sum\limits_{l \in L} K_lx_{jl},\quad \forall~j \in J,\label{capfacility}\\
&x_{jl} \in \{0,1\}, \quad \forall~j \in J,~l \in L,\quad R_j \geq
0, \quad \forall~j \in J.
\end{align}
\end{subequations}
Here, $\vxi(\omega_s)=(\vd^s,\vgamma^s,\vc^s,\vnu^s,\pi^s)$ is the
vector of the random parameter realizations under scenario $s\in S$.
Given the first-stage decision vectors $\vx$ and $\vR$, the optimal
second-stage objective value under scenario $s$,
$Q(\vx,\vR,\vxi(\omega_s))$, is calculated by solving the following
second-stage problem under scenario $s$:
\begin{subequations}
\label{app:second}
\begin{align}
\min \quad&  \sum\limits_{j \in J}\sum\limits_{i\in M^s_{j}
\setminus \{j\}}c^s_{ji}y^s_{ji}+\sum\limits_{i \in I} \pi^su^s_{i}
\label{2stage-obj}
\\
 \text{s.t.}\quad&
u^s_{i} \geq d^s_{i}-\sum\limits_{j\in N_i^s}y^s_{ji},\quad
\forall~i\in
I\setminus J,\label{demandshortage1}\\
& u^s_{j} \geq d^s_{j} + \sum\limits_{i \in M_j^s \setminus
\{j\}}y^s_{ji} -(R_j\gamma^s_j+ \sum\limits_{i \in
N_j^s \setminus \{j\}}y^s_{ij}),\quad \forall~j \in J,\label{demandshortage2}\\
&\sum\limits_{i \in M_j^s \setminus \{j\}}y^s_{ji} \leq
R_j\gamma^s_j,\quad \forall~j \in
J,\label{demandshortage3}\\
&\sum\limits_{i \in N_j^s\setminus \{j\}}y^s_{ij} \leq
(1-\sum\limits_{l \in L} x_{jl})d_j^s,\quad \forall~j \in
J,\label{demandshortage4}\\
& u^s_{\max} \ge \frac{u^s_{i}}{d_{i}^s},\quad \forall~i \in I,\label{maxshortage}\\
&y^s_{ji} \geq 0 \quad \forall~j \in J,~i\in M^s_{j} \setminus
\{j\},\\
&u^s_{i}\geq 0\quad \forall~i \in I, \quad u^s_{\max} \ge 0.
\end{align}
\end{subequations}
The objective, defined by \eqref{1stage-obj} and \eqref{2stage-obj},
is to minimize the expected total cost of opening facilities,
purchasing and shipping the relief supplies, and demand shortages.
As discussed in Section \ref{sec:general-model}, constraint
\eqref{con:cvarapp} ensures that the decision-based random outcome
vector $\vG(\vx,\vR,\vxi(\omega))$ is preferable to the benchmark
outcome $\vZ$ according to the multivariate polyhedral CVaR
relation. Here $\vG(\vx,\vR,\vxi(\omega_s)) = \hat{\vg}_s(\vx, \vR,
\vy^s, \vu^s, u_{\max}^s)$ is a two-dimensional vector, the
components of which are given by
\begin{align}
  &\vg_s^1(\vx,\vR,\vy^s,\vu^s,u_{\max}^s)=u^s_{\max},\label{1-measure} \\
  &\vg_s^2(\vx,\vR,\vy^s,\vu^s,u_{\max}^s)= \dfrac{\sum\limits_{i \in
I}\operatorname{ATS}^s_i}{\sum\limits_{i \in I}\max\limits_{j\in
N^s_{i}\setminus \{i\}}  \nu_{ij}^s }, \text{where }
\operatorname{ATS}^s_i =\sum\limits_{j\in N^s_{i}\setminus \{i\}}
\dfrac{\nu_{ij}^s y^s_{ji} }{ d_{i}^s}.
\label{2-measure}
\end{align}
The maximum proportion of unsatisfied demand \eqref{1-measure} is
crucial for equity in terms of supply allocation, whereas the second
measure \eqref{2-measure}, the total delivery amount-based average
travel time score (ATS), is related to responsiveness. The
delivery-based unit average travel time to demand node $i$ could
ideally be calculated by replacing $d_{i}^s$ with the corresponding
exact delivery amount (i.e., $y_{ji}^s$) in \eqref{2-measure}; we
prefer not to use this calculation to avoid nonlinear terms.
Moreover, \eqref{2-measure} relies on a scaling with an upper bound
($\max\limits_{j\in N^s_{i}\setminus \{i\}} \nu_{ij}^s $) of
$\operatorname{ATS}_i$; this scaling ensures that both measures take
values between 0 and 1 and prevents biased solutions due to the
differences in the magnitude of the outcomes. We also note that
considering a set of scalarization vectors in \eqref{con:cvarapp} is
essential to deal with the ambiguities and inconsistencies in the
weight vectors. Eliciting the relative weights of even a single
decision maker/expert can be challenging \citep[for a related review, see,
e.g.,][]{Liu15}, and is further exacerbated if multiple experts
are involved as in the case of humanitarian logistics applications.

We next elaborate on the remaining constraints of the proposed
model. Constraints (\ref{numfacility}) ensure that at most one
facility is opened at any node $j \in J$. Constraints
(\ref{capfacility}) guarantee that the inventory levels at open
facilities do not exceed the corresponding capacity limits and there
is a facility located at node $j$ if its inventory level is
positive. Constraints
\eqref{demandshortage1}-\eqref{demandshortage4} correspond to a
restricted version of the classical network flow formulation, which
provides a more structured flow considering the coverage issues.
More specifically, these constraints enforce that the nodes without
any pre-stocked inventory receive service directly from the open
facilities that are sufficiently close (according to the upper bound
on travel times). Constraints (\ref{demandshortage3}) ensure that
the amount delivered from a facility node does not exceed its
available inventory level and outgoing flow is not possible if there
is no facility located. Constraints (\ref{demandshortage4}) assure
that there is no delivery to a node if a facility is located at that
node, and the amount of delivery is limited by its demand,
otherwise. Constraints (\ref{maxshortage}) calculate the maximum
proportion of demand shortage.

\ignore{Constraints (\ref{demandshortage1}) and
(\ref{demandshortage2}) are the flow balance inequalities and
provide a positive shortage amount for a node if its demand is not
satisfied....The rest of the constraints are for the non-negativity
and binary restrictions on the decision variables.}

We develop computationally effective solution methods for the model
\eqref{app:first}-\eqref{app:second} by applying our cut
generation-based algorithms presented in Section
\ref{sec:general-model}. Observing that the original second-stage
problem \eqref{app:second} is always feasible, it is interesting to
note that our scenario decomposition-based algorithm may need to
generate the Benders feasibility cuts \eqref{feascut}. This is the
case since the algorithm solves iteratively modified versions of
\eqref{app:second} with the constraints of type \eqref{C:p1}, which
may lead to infeasibility in a second-stage subproblem.

\section{Computational Study}\label{sec:comp-study}

In this section, we illustrate the computational effectiveness of
the two types of algorithms proposed in Section \ref{sec:general-model} by implementing them to solve the
two-stage model \eqref{app:first}-\eqref{app:second} under
varying parameter settings. We first describe how the test problem
instances are generated and then provide numerical results.

\ignore{(proposed in Sections \ref{dcg-def} and \ref{dcg-sd}) ...we
just refereed to it above}

\ignore{Cvar daha kolay gibi argumanlar yazmanin uygun olmadigini
farkettim, cut generation'da durum kesinlikle boyle ama overall
algoritma icin bunu iddia edip referans veremiyoruz...CVaR iliskisi
SSD gibi conservative degil diye model kisminda not ettim, yani biz
practice'da relaxation olarak CVaR daha anlamli, onu sectik deyip
burada da direk onerdigimiz disaster modelini cozduk deyip
gecistirmeye calistim...bu acidan dusunup modelde sunu not ettim:
``This CVaR-based relation is preferred, since it is a natural
relaxation of its SSD-counterpart which typically results in very
demanding constraints that can be excessively hard to satisfy in
practice.'' ...``we restrict our attention to the CVaR-constrained
model (\ref{GeneralForm1}) to highlight the computational
effectiveness of our solution methods.'' cumlesini silip altini
cizmedim yani}

We used the data sets of \citet{OzgunThesis16}, which are based on a
case study concerning the threat of hurricanes in the Southeastern
part of the United States \citep{Rawls2010}; in this case study, the
region is represented by a network with 30 nodes and 55 arcs, and
only a single set of hurricane scenarios with a cardinality of 51 is
considered. In order to generate instances with a larger number of
scenarios, \citet{Hong2015} propose a scenario generation method
which takes into account the dependency structures inherent in
disaster relief networks. Their approach randomly identifies the
link degradation and node damage levels depending on the proximity
to the center of the disaster, and generates the realizations of the
random input parameters accordingly. \citet{OzgunThesis16} follows
this elaborate scenario generation approach, but makes the necessary
modifications according to their assignment-based formulation of the
second-stage problem. More specifically, the author ignores 55 links
of the original network and introduces the links according to the
coverage-based responsiveness requirement. In our computational
study, we consider the most dense (fully-connected) network
structure to distribute water (in units of 1000 gallons). We refer
the reader to \cite{Hong2015} and \cite{OzgunThesis16} for further
details on the values of the relevant parameters, such as cost,
facility capacity limits, base values of the travel times and water
demand, and the random distribution information, etc. Finally, we
note that all scenarios are assumed to be equally likely.

We next discuss the additional issues that arise due to the
different features of our new risk-averse model. The stochastic
benchmarking constraints of interest additionally require us to
specify a benchmark random outcome vector and a scalarization set.
For illustrative purposes, we obtain a benchmark vector by
identifying a reasonable and feasible benchmark first-stage solution
and calculating the performance measures of interest associated with
the corresponding optimal second-stage decisions. In particular, we
consider the first-stage solution, where a single large-sized
facility with the maximum possible inventory level (5394 units) is
located at Houston, Baton Rouge, Biloxi, Atlanta and Orlando. In
real life applications, a benchmark decision
vector 
can be constructed by considering the existing practices of relief
organizations. In addition, the scalarization set is assumed to have
an ordered preference structure as follows: $\C = \{\vc \in \R_+^2:
c_1+c_2 = 1, c_2 \geq c_1\}$, where $c_1$ and $c_2$ correspond to
the relative importance of the demand satisfaction criterion
\eqref{1-measure} and responsiveness criterion \eqref{2-measure},
respectively. This particular scalarization set description can be
easily modified according the preferences of the decision makers.

\ignore{This particular selection of scalarization set asserts that
the responsiveness criterion is at least as important as the demand
satisfaction criterion; }

\vspace{0.5cm}

\noindent{\bf Computational Performance of the Solution Methods.} We
evaluate the computational performance of the proposed solution
algorithms in terms of the run times and the number of operations
performed at each iteration.

All experiments are performed on a single thread of a Windows server
with Intel(R) Xeon(R) CPU E5-2630 processor at 2.40 GHz and 32 GB of
RAM using Java and Cplex 12.6.0. CPLEX is invoked with a time limit
of 3600 seconds. MIP gap and feasibility tolerances are set to
$10^{-5}$ and $10^{-9}$, respectively; otherwise, we use the default
CPLEX settings. The results are obtained for the instances with
$\alpha = 0.99, 0.95, 0.90$ and up to 1500 scenarios. For each
setting, we report the average over three instances. We implement
the cutting plane algorithms using the \textit{lazy constraint
callback} function of CPLEX. Considering the fact that using
callbacks prevents CPLEX from utilizing dynamic search, we specify
the MIP search method for both algorithms as \textit{traditional
branch-and-cut}. Following an acceleration technique suggested in
\cite{BL16}, we initialize the RMP of the scenario
decomposition-based algorithm with optimality and feasibility cuts
associated with an initial first-stage solution to benefit from the
possible improvement obtained from the automatically created CPLEX
cuts. In particular, our initial solution locates two facilities at
Mobile and Orlando with the smallest size. Additionally, we employ
the \textit{user cut callback} function of CPLEX for adding the
Benders cuts at the root node for the fractional incumbent
solutions. The \textit{usercut callback} is invoked at each
iteration at the root node until the improvement in the relative MIP
gap is less than $10^{-6}$ for the last five consecutive iterations.

\ignore{
This
implementation allows the solver to set a better initial lower bound
at the root node.
}
\ignore{
In addition, when the scenario
decomposition-based algorithm encounters an infeasible second-stage
subproblem, it directly adds the corresponding Benders feasibility
cut the RMP. On the other hand, in case an optimal solution of the
subproblem is found, the algorithm adds the corresponding Benders
optimality cut only if it improves the current upper bound.
}

In Table \ref{tab:tab2}, we report several statistics on the
performance of the delayed cut generation algorithm for DEF
(DCG-DEF) and its scenario decomposition-based counterpart (DCG-SD).
Under the ``Time" columns,  the solution times for only the
instances solved within the time limit are reported. \ignore{ Each
replication such that the algorithm terminated with a feasible
solution due to time limit is indicated with a
{\color{blue}$^\dagger$} under the column
``Total/[\%gap]"\comment{bunlari burada soylemeye gerek var mi,
normalde bu isaretler tablo altindan aciklanir}. The replications
for which no feasible solutions can be obtained within the time
limit are marked with a {\color{red}$^*$}.} For both solution
algorithms, we report the total run time and the time spent solving
the separation problems  under columns ``Total/[\%gap]",
 and ``Sep.", respectively. If the algorithm
terminated at the time limit with a feasible solution for some
replications of the instance, then the average of the relative
optimality gaps are reported under the column ``Total/[\%gap]",
inside square brackets. The relative optimality gap reported by
CPLEX corresponds to the relative gap between the objective function
value for the best available integer solution ($z^{IP}$) and the
best LP relaxation bound at the time of termination ($z^{LP}$), i.e.
$100|z^{IP}-z^{LP}|/z^{IP}$. For both algorithms, we report the
average number of separation problems solved and the average number
of scalarization vectors added under columns ``Sep." and ``$L$'',
respectively. For DCG-SD, we report the average total number of cuts
(\eqref{cvarcut}, \eqref{feascut}, and \eqref{optcut}) added to the
master problem under column ``Cuts". We do not report the number of
cuts for DCG-DEF, because in this case, constraints of type
\eqref{DEFForm-cvar-const1} and \eqref{DEFForm-cvar-const2} are
added when a new scalarization vector is included in the RMP. Hence,
the total number of cuts added to the RMP for DCG-DEF is $(|S|+1)L$,
where $|S|$ is the number of scenarios and $L$ is the number of
scalarization vectors generated.

The results  show that DCG-SD outperforms DCG-DEF in terms of the
computation times and number of instances that can be solved to
optimality. Out of 63 instances, DCG-DEF is able to find an optimal
solution for only 35 instances and it terminates with a feasible
solution after one hour for seven instances. On the other hand,
DCG-SD provides at least a feasible solution for all instances even
if it fails to prove optimality for five of them. Considering the
instances that can be solved to optimality by both methods, it can
be seen that DCG-SD proves optimality in a shorter amount of time
than DCG-DEF.

We observe that unlike the single-stage multivariate
CVaR-constrained problems in \cite{Noyan13}, the separation problem
is no longer the bottleneck taking over 95\% of the solution time.
In fact, for DCG-DEF, the  time spent on separation is negligible
when compared to the overall solution time, with a small number of
separation problems solved. On the other hand, a large  number of
separation problems are solved in DCG-SD, hence it is important to
use the strongest formulations available for the separation problem
(i.e., those provided in \citealp{Liu15} for the equiprobable scenario
case) to reduce the time spent on generating the scalarization
vectors. Nevertheless, a few instances cannot be solved with DCG-SD
within the time limit mainly due to the long solution times of the
separation MIPs. However, for these instances, our algorithm is able
to find a feasible solution during its course, and hence it provides
optimality gaps at termination. To avoid long solution times, one
can implement heuristic separation methods. In addition, if
$\varepsilon$-feasible solutions are satisfactory, then the
violation thresholds can be updated accordingly.


\begin{table}[ht]
  \small  \centering
  \caption{Performance of the proposed delayed
cut generation  algorithms ($|I|=|J|=30$)}
      \begin{tabular}{cc|cc|cc|cc|ccc}
    \toprule
          &       & \multicolumn{4}{|c|}{DCG-DEF}                & \multicolumn{5}{c}{DCG-SD} \\
   \hhline{|~~|-|-|-|-|-|-|-|-|-|}
          &       & \multicolumn{2}{|c|}{Time (s)} & \multicolumn{2}{c|}{Number}  & \multicolumn{2}{c|}{Time (s)}  & \multicolumn{3}{c}{Number} \\
          \hhline{|~~|-|-|-|-|-|-|-|-|-|}
    $\alpha$ & \# Sce. & Total / [\%gap] & Sep.  & Sep. & $L $       & Total / [\%gap]  & Sep.        & Cuts   & Sep.  & $L$ \\ \hhline{|-|-|-|-|-|-|-|-|-|-|-|} 
          & 200   &   {704.9}  &   {0.2}  &   {4.3} &   {2.3}  &   {46.8} &    {3.1} &    {6101.3}  &   {62.0} &   {3.0} \\
          & 300   &   {1483.5}  &   {0.7}   &   {5.3} &   {3.0}   &   {108.7} &    {5.1}   &   {9697.3} &   {78.0} &   {3.3} \\
          & 400   &   {2622.4}   &   {1.2}  &   {5.3} &   {3.7}   &   {351.5} &    {13.9}   &   {18569.7}  &   {122.0} &   {5.3} \\
  0.99         & 500   &   {3071.3}  &   {0.8}   &   {4.0} &   {3.0}   &   {363.3}  &   {17.7}   &   {20805.0} &    {105.3} &   {4.3} \\
          & 600   &   {${\color{blue} [0.2]^\dagger}{\color{red}^{**}}$}  &   {-}  &   {-} &   {-}   &   {662.1}  &   {42.4}   &   {27068.7} &   {157.0} &   {5.0} \\
          & 800   &   {${\color{red}^{***}}$} &   {-} &   {-} &   {-}   &   {987.2} &    {51.7}   &   {31994.7}  &   {129.7} &   {4.7} \\
          & 1000  &   {${\color{red}^{***}}$} &   {-}   &   {-} &   {-}   &   {1594.2} &    {110.3}   &   {43934.0} &   {129.3} &   {5.3} \\
          & 1500  &   {${\color{red}^{***}}$} &   {-}    &   {-} &   {-}   &   {3450.6 ${\color{blue}[0.4]^{\dagger \dagger}}$} &   {506.0}   &   {62436.0}&   {172.0} &   {4.0} \\\hline
          & 200   &   {408.9} &   {0.7}   &   {2.7} &   {0.7}   &   {148.6} &    {17.4}   &   {9524.7} &    {137.0} &   {4.0} \\
          & 300   &   {1476.4}  &   {1.6}   &   {3.3} &   {1.0}   &   {277.6} &    {26.1}   &   {13280.0} &   {128.7} &   {4.0} \\
          & 400   &   {3392.7} &   {2.1}   &   {3.7} &   {1.7}   &   {536.8} &    {45.2}   &   {20633.3} &   {162.0} &   {4.0} \\
   0.95         & 500   &   {2753.2 ${\color{blue}[0.2]^{\dagger}}$} &   {1.4}   &   {2.5} &   {1.5}   &   {671.9} &   {150.6}   &   {23045.0}  &   {161.3} &   {3.3} \\
          & 600   &   {${\color{blue}[0.1]^{\dagger \dagger}}{\color{red}^*}$} &   {-}   &   {-} &   {-}   &   {878.0} &   {156.2}   &   {26151.3} &   {141.0} &   {3.3} \\
          & 800   &   {${\color{red}^{***}}$} &   {-}  &   {-} &   {-}   &   {1616.4} &    {126.4}   &   {33607.7} &   {115.7} &   {3.0} \\
          & 1000  &   {${\color{red}^{***}}$} &   {-}    &   {-} &   {-}   &   {1857.7} &    {274.6}   &   {42229.3}  &   {140.3} &   {3.3} \\
          & 1500  &   {${\color{red}^{***}}$} &   {-}    &   {-} &   {-}   &   {3390.1 ${\color{blue}[1.4]^{\dagger \dagger}}$}  &   {867.2}   &   {52294.0}  &   {120.0} &   {2.0} \\ \hline
          & 200   &   {397.4}  &   {1.0}   &   {2.7} &   {0.7}   &   {303.5} &    {36.5}   &   {11781.0} &   {180.0} &   {4.3} \\
   0.90        & 300   &   {1128.5} &   {0.6}   &   {3.0} &   {0.7}   &   {352.5} &   {48.7}   &   {13152.7} &   {111.3} &   {3.0} \\
          & 400   &   {2263.0}  &   {1.6}   &   {4.0} &   {1.3}   &   {623.9} &    {89.3}   &   {19482.3}  &   {130.3} &   {3.3} \\
          & 500   &   {1660.2 ${\color{blue}[0.3]^\dagger}$} &   {2.8}   &   {2.0} &   {1.0}   &   {1392.1} &   {228.6}   &   {25358.0}  &   {150.3} &   {3.3} \\
          & 600   &   {2264.3 ${\color{blue}[0.1]^{\dagger \dagger}}$}  &   {2.2}   &   {2.0} &   {1.0}   &   {857.5 ${\color{blue}[7.9]^\dagger}$} &   {229.6}   &   {30051.0}  &   {153.5} &   {3.0} \\
    \bottomrule
    \multicolumn{1}{r}{{\color{blue} $\dagger$}:}& \multicolumn{10}{l}{Each dagger sign indicates one instance hitting the time limit with an integer feasible solution.}\\
     \multicolumn{1}{r}{{\color{red} $*$}:}& \multicolumn{10}{l}{Each asterisk sign indicates one instance hitting the time limit with no integer feasible solution.}
    \end{tabular}%
\label{tab:tab2}
\end{table}%

We close this section by noting that the problem sizes we consider
are significantly larger (in terms of first- and second-stage
variables and constraints and/or the number of scenarios) when
compared to the data sets in \cite{Noyan13} and \cite{DW15} for
related optimization problems with multivariate stochastic
benchmarking constraints. In addition, our first-stage problem
involves discrete decisions, which results in the most difficult
data sets considered for this problem class to date. In conclusion,
our computational experiments showcase the scalability of the
proposed scenario decomposition-based method with respect to the
number of scenarios.

\ignore{the scalability of the proposed method..iki algoritma var
deyip durunca bu tekil garip oluyor sanki, asil scale eden de SD
olan...}

\ignore{

\section{Conclusions}\label{sec:conclusion}

We consider two-stage optimization problems under multivariate
stochastic benchmarking preference constraints. For finite
probability spaces, we give an exact algorithmic framework that
obtains optimal solutions to the problems of interest in finitely
many iterations. We highlight the case where the multivariate
stochastic preference relation is based on the CVaR measure, but
show that our framework is also applicable to the SSD-based
counterpart. A desirable feature of our algorithm is that it
exploits the decomposable structure of the second-stage problems and
the stochastic preference constraints. We demonstrate the
performance of our algorithm on a humanitarian relief network design
problem. Our computational experiments showcase the scalability of
the proposed method with respect to the number of scenarios.

}

\section*{Acknowledgements}
Simge K\"{u}\c{c}\"{u}kyavuz and Merve Merakl{\i} are supported, in
part, by National Science Foundation Grant 1537317.  Nilay Noyan
acknowledges the support from The Scientific and Technological
Research Council of Turkey under grant \#115M560.

\appendix
\section*{Appendix A. Duality Results} \label{sec:app}
In this section, we develop duality formulations and optimality
conditions for the important special case of the multivariate
CVaR-constrained two-stage model
\eqref{GeneralForm1}-\eqref{second-stage-problem}, where the mapping
$f$ is linear, the set $\X $ is a polyhedral set, and the
first-stage decisions are continuous. Let $\X=\{\vx \in
\R_+^{n_1}~:~A\vx \leq \vb \}$ and $f(\vx)= \vf^\top\vx$ for some
matrix $A \in \R^{m_1 \times n_1}$, and the vectors $\vf\in
\R^{n_1}$ and $\vb\in \R^{m_1}$. Then, the model of interest becomes
\begin{subequations}\label{general_linear_primal}
\begin{align}
(\operatorname{LinearP})\quad \min~~&\vf^\top\vx+\sum_{s \in S} p_s \vq^\top_s  \vy_s  \\
\text{s.t.}~~
&A \vx \leq \vb,\quad \vx\in \R_+^{n_1},\label{const:X}\\
&\eqref{c:cvar}, ~\eqref{eq:def-couple}-\eqref{eq:def-couple2}.
\end{align}\end{subequations}
For a finite set
$\tilde{\C}=\{\tilde{\vc}_{(1)},\ldots,\tilde{\vc}_{( L)}\}$  we
consider the following LP, referred to as
$(\operatorname{FiniteP}(\tilde{\C}))$:$$\min \{\vf^\top\vx+\sum_{s
\in S} p_s \vq^\top_s
\vy_s~:~\eqref{DEFForm-cvar-const1}-\eqref{DEFForm-cvar-const3}
 \text{
(with } \bar{L} \text{ replaced by } L),\eqref{const:X},~\veta
\in \R^{L},~
\eqref{eq:def-couple}-\eqref{eq:def-couple2}\}.$$The next lemma
shows that $(\operatorname{FiniteP}(\tilde{\C}))$ is equivalent to
$(\operatorname{LinearP})$ for suitable choices of $\tilde{\C}$; it
is a simple consequence of Proposition \ref{pro:DEFform}.
\begin{lemma}\label{prop_linear}  Let $\hat{\C}= \{\hat
\vc_{(1)},\ldots,\hat{\vc}_{(L)}\}$ as defined in Proposition
\ref{pro:DEFform}, and assume that the finite set $\tilde{\C}$
satisfies $\hat{\C}\subset\tilde{\C}\subset \C$. Then a vector
$(\vx,\vy)$ is a feasible (optimal) solution of
$(\operatorname{LinearP})$ if and only if
$(\vx,\vy,\veta^{(\vx,\vy)},\vw^{(\vx,\vy)})$ is a feasible
(optimal) solution of $(\operatorname{FiniteP}(\tilde{\C}))$, where
$\eta^{(\vx,\vy)}_{\ell}=\var_{\alpha}(\tilde{\vc}_{(\ell)}^\top\hat{\vG}(\vx,\vy))$
and
$w^{(\vx,\vy)}_{i\ell}=[\eta^{(\vx,\vy)}_{\ell}-\tilde{\vc}_{(\ell)}^\top\hat{\vg}_s(\vx,\vy_s)]_+$.
\end{lemma}
The LP formulation $(\operatorname{FiniteP}(\tilde{\C}))$ does not
immediately offer a tractable solution approach, however it leads to
a natural delayed cut generation algorithm as discussed in Section
\ref{sec:general-model}. Moreover,
$(\operatorname{FiniteP}(\tilde{\C}))$ provides an important direct
way to derive strong duality results and optimality conditions,
presented below.
First, recall the following notation:
$\hat{\vg}_s(\vx,\vy_s)=(\vg_s^1(\vx,\vy_s),\vg_s^2(\vx,\vy_s),\ldots,\vg_s^d(\vx,\vy_s))^\top,~s\in
S,$ and $\vg_s^i(\vx,\vy_s) = \vbg_s^i \vx + \vtg_s^i \vy_s$,~$i\in
\{1,\ldots,d\},~s\in S$. By abuse of notation, we introduce the
random matrices $\bar{\vg}~:~\Omega \rightarrow\R^{d\times n_1}$ and
$\tilde{\vg}~:~ \Omega \rightarrow \R^{d\times n_2}$, where
$\bar{\vg}^\top(\omega_s)=\bar {\vg}_s^\top$ and
$\tilde{\vg}^\top(\omega_s)=\tilde{\vg}_s^\top$ with the $i$th
columns  defined by $\vbg_s^i$ and $\vtg_s^i$, respectively.
Denoting the set of all finitely supported finite non-negative
measures on the scalarization polyhedron $\C$ by $\M_+^F(\C)$, we
obtain the following dual problem to $(\operatorname{LinearP})$:
\begin{subequations}\label{general_linear_dual}
\begin{align}
(\operatorname{LinearD})\quad \max~~&
-\int_\C \cvar_{\alpha}(\vc^\top\vZ)\,\mu(\dd\vc)-\vlambda^\top\vb +\E[\vpi^\top \vh(\omega)]\\
\text{s.t.}~~& \E(\nu)=\mu,\label{dualfeascond1}\\
&\nu(\omega_s) \leq \frac{\mu}{1-\alpha},\quad ~\forall~s\in S,\label{dualfeascond2}\\
&\E[\vpi^\top T(\omega)]-\E\left(\int_\C \vc^\top\bar{\vg}\,\nu(\dd\vc)\right)\leq \vf^\top+\vlambda^\top A,\label{dualfeascond3}\\
&\vpi^\top_sW_s-\int_\C \vc^\top\tilde{\vg}_s[\nu(\omega_s)](\dd\vc)\leq
\vq_s^\top,\quad ~\forall~s\in S,\label{dualfeascond4}\\
&\vlambda\in\R_+^{m_1},\quad
\mu\in\M_+^F(\C),\quad\nu:\Omega\rightarrow\M_+^F(\C),\quad \vpi_s \in
\R_+^{m_2},~s\in S.
\end{align}\end{subequations}
The proof of the next duality theorem, which is similar to that of
Theorem 3 in \citet{Noyan13}, is omitted here for the sake of
brevity. It mainly relies on Lemma \ref{prop_linear} and the
following facts for a finitely supported measure $\mu\in\M_+^F(\C)$:
$\supp(\mu)=\left\{\vc\in \C~:~\mu\left(\{\vc\}\right)>0\right\}$ and
$\int_\C  u(\vc)\,\mu(\dd\vc)=\sum_{\vc\in\supp(\mu)}
u(\vc)\mu\left(\{\vc\}\right)$ for a function $u:\C\rightarrow\R$.
\begin{theorem}\label{thm:strong-duality}The problem $(\LP)$ has a
finite optimum value if and only if $(\LD)$ does, in which case the
two optimum values coincide. In addition, a feasible solution
$(\vx,\vy)$ of $(\LP)$ and a feasible solution
$(\vlambda,\mu,\nu,\vpi)$ of $(\LD)$ are both optimal for their
respective problems if and only if the following complementary
slackness conditions hold: \small{
$$\begin{array}{llr}
(i)& \supp(\mu) \subset\quad \left\{\vc~:~
\cvar_{\alpha}(\vc^\top \hat{\vG}(\vx,\vy))=\cvar_{\alpha}(\vc^\top\vZ)\right\},&\\
(ii)& \supp(\nu(\omega_s))\subset\quad \left\{\vc~:~
\var_{\alpha}(\vc^\top\hat{\vG}(\vx,\vy))\leq\vc^\top\hat{\vg}_s(\vx,\vy_s)\right\},&\quad s\in S,\\
(iii) & \supp(\frac{\mu}{1-\alpha}-\nu(\omega_s))\subset\quad
\left\{\vc~:~
\var_{\alpha}(\vc^\top\hat{\vG}(\vx,\vy))\geq\vc^\top\hat{\vg}_s(\vx,\vy_s)\right\},&\quad s\in S,\\
(iv) & \vlambda^\top(\vb-A\vx)=0,&\\
(v)&\vpi^\top_s(T_s\vx+W_s\vy_s - \vh_s)=0,&\quad s\in S,\\
(vi)&(\vf^\top+\vlambda^\top A- \E[\vpi^\top T(\omega)]+\E\left(\int_\C \vc^\top\bar{\vg}\,\nu(\dd\vc)\right))\vx=0,&\\
(vii)&(\vq_s^\top-\vpi^\top_sW_s+\int_\C
\vc^\top\tilde{\vg}_s[\nu(\omega_s)](\dd\vc))\vy_s=0,&\quad s\in S.
\end{array}$$}
\end{theorem}
We remark that our dual formulation is analogous to Haar's dual for
semi-infinite linear programs \citep[see, e.g.,][]{Bonnans00}. The finite
representation of the multivariate CVaR relation (appearing in
Proposition \ref{pro:DEFform}) proves to be useful to obtain strong
duality results and optimality conditions directly from linear
programming duality without the need for constraint qualifications.

\ignore{For the proof, see the proof of Theorem 3 of
\citet{Noyan13}; similar approach should work}


\vspace{0.5cm}

\noindent{\bf Lagrangian Duality.}
The duality results established in Theorem \ref{thm:strong-duality}
have a natural Lagrangian interpretation. In fact, measures on $\C$
are a natural choice to use as Lagrange multipliers, since the CVaR
constraints in \eqref{c:cvar} are indexed by the scalarization set
$\C$. Accordingly, we define the Lagrangian function
$\cL:(\X,\Y)\times\M_+^F(\C)\rightarrow\R$ as follows:
\begin{equation}\label{lagfunc}\cL(\vx,\vy,\mu)=\vf^\top\vx+\sum_{s \in S} p_s \vq^\top_s  \vy_s
+\int_\C \cvar_\alpha(\vc^\top \hat{\vG}(\vx,\vy))\,\mu(\dd
\vc)-\int_\C \cvar_\alpha(\vc^\top \vZ))\,\mu(\dd
\vc).\end{equation} Here, by abuse of notation, $\vy \in
\R_+^{|S|n_2}$ represents the collection of decision vectors $\vy_s
\in \R_+^{n_2},~s\in S$, and $\vy \in \Y$ refers to $\vy_s \in
\Y(\vx,\vxi(\omega_s)),~s\in S$. For the ease of exposition, we
assume that the feasible sets $\X$ and $\Y$ are compact; this is a
very mild assumption due to the finiteness of the first-stage and
second-stage objective function values. Then,
$(\operatorname{LinearP})$ is equivalent to
$\min\limits_{(\vx,\vy)\in
(\X,\Y)}\max\limits_{\mu\in\M_+^F(\C)}\cL(\vx,\vy,\mu)$, while the
corresponding Lagrangian dual problem is given by
\begin{equation}\label{Ldual-rho}(\operatorname{LagrangianD})\max\limits_{\mu\in\M_+^F(\C)}\min\limits_{(\vx,\vy)\in
(\X,\Y)}\cL(\vx,\vy,\mu).\end{equation}For any feasible solution of
$(\LP)$ and any $\mu\in\M_+^F(\C)$, $\cL(\vx,\vy,\mu)\leq
\vf^\top\vx+\sum_{s \in S} p_s \vq^\top_s \vy_s$. The weak duality
immediately follows: $\OBF_P \geq \OBF_D$, where $\OBF_P$ and
$\OBF_D$, respectively, denote the optimum objective values of
$(\LP)$ and $(\LD)$. The next theorem, which is a consequence of
Theorem \ref{thm:strong-duality}, provides the strong duality result
and optimality conditions.
\begin{theorem}If the primal problem $(\operatorname{LinearP})$ has an optimal solution,
then the dual problem $(\operatorname{LinearD})$ also has an optimal
solution, and the optimal objective values coincide. A primal
feasible solution $(\vx^*,\vy^*)$ and a dual feasible solution
$\mu^*$ are simultaneously optimal if and only if they satisfy the
equations
\begin{align}\int_\C \cvar_\alpha(\vc^\top\hat{\vG}(\vx^*,\vy^*))\,\mu^*(\dd\vc)&=\int_\C \cvar_\alpha(\vc^\top
\vZ)\,\mu^*(\dd\vc),\label{lag-cond1}\\\cL(\vx^*,\vy^*,\mu^*)&=\min\limits_{(\vx,\vy)\in
(\X,\Y)}\cL(\vx,\vy,\mu^*).\label{lag-cond2}\end{align}
\end{theorem}

\ignore{Simply modified the proof of Theorem 4.4 of \citet{Noyan16}}
\begin{proof}
Let us first show the following claim: \emph{If the equations
\eqref{lag-cond1}-\eqref{lag-cond2} hold for some primal feasible
$(\vx^*,\vy^*)$ and dual feasible  $\mu^*$, then these solutions are
simultaneously optimal with coinciding objective values.} According
to \eqref{lag-cond1} and \eqref{lag-cond2}, we have
\begin{align*}\OBF_P\leq \vf^\top\vx^*+\sum_{s \in S} p_s \vq^\top_s
\vy^*_s&= \cL(\vx^*,\vy^*,\mu^*)=\min\limits_{(\vx,\vy)\in
(\X,\Y)}\cL(\vx,\vy,\vmu^*)\\ &\leq
\max\limits_{\mu\in\M_+^F(\C)}\min\limits_{(\vx,\vy)\in
(\X,\Y)}\cL(\vx,\vy,\mu)=\OBF_D.\end{align*}On the other hand, the
weak duality implies that $\OBF_P \geq \OBF_D$, which proves the
claim.

\emph{We next prove the claim that if $(\vx^*,\vy^*)$ is primal
optimal and $\mu^*$ is dual optimal, then they satisfy the equations
\eqref{lag-cond1}-\eqref{lag-cond2}. To this end, we first show that
for any given primal optimal solution $(\vx^*,\vy^*)$ there exists a
corresponding dual feasible solution $\bar \mu$ such that
\eqref{lag-cond1} and \eqref{lag-cond2} are satisfied for the choice
$\mu^*=\bar \mu$.} Let us consider an optimal solution
$(\vx^*,\vy^*)$ of $(\LP)$. By Theorem \ref{thm:strong-duality}
there exists an optimal solution $(\bar
\vlambda,\bar\mu,\bar\nu,\bar\vpi)$ of $(\LD)$ with the same optimal
objective value. According to the complementary slackness condition
(i), the equality
$\cvar_\alpha(\vc^\top\hat{\vG}(\vx^*,\vy^*))=\cvar_\alpha(\vc^\top
\vZ)$ holds on the support of the measure $\bar\mu$, which implies
\eqref{lag-cond1}. \emph{To show that \eqref{lag-cond2} also holds,
we next prove that, substituting $\mu^*=\bar \mu$,
$\cL(\vx,\vy,\bar\mu)\geq \cL(\vx^*,\vy^*,\bar\mu)$ is valid for any
primal feasible solution $(\vx,\vy)$ (and equality holds for
$(\vx^*,\vy^*)$)}. For ease of exposition, let $U(\vx,\vy,\bar
\mu)=\int_\C \cvar_\alpha(\vc^\top \hat{\vG}(\vx,\vy))\,\bar\mu(\dd
\vc)$ and  $U(\vZ,\bar \mu)=\int_\C \cvar_\alpha(\vc^\top
\vZ))\,\bar\mu(\dd \vc)$. Then, using the definition of the
Lagrangian function \eqref{lagfunc} and the dual feasibility
conditions \eqref{dualfeascond3}-\eqref{dualfeascond4} we obtain the
following relations: \begin{align} &\cL(\vx,\vy,\bar\mu)+U(\vZ,\bar
\mu)=\vf^\top\vx+\sum_{s \in S} p_s
\vq^\top_s \vy_s +U(\vx,\vy,\bar \mu)\geq-\bar\vlambda^\top A \vx \notag + \E[\bar \vpi^\top T(\omega)\vx]\\
&-\E\left(\int_\C \vc^\top\bar{\vg}\vx\,\bar\nu(\dd\vc)\right)
+\sum_{s \in S} p_s\left(\bar
\vpi^\top_sW_s\vy_s-\int_\C \vc^\top\tilde{\vg}_s\vy_s[\bar\nu(\omega_s)](\dd\vc)\right)+U(\vx,\vy,\bar \mu)=-\bar\vlambda^\top A \vx\notag\\
&+\sum_{s \in S} p_s\bar\vpi^\top_s(T_s\vx+W_s\vy_s)
-\E\left(\int_\C \vc^\top\bar{\vg}\vx\,\bar\nu(\dd\vc)\right)
-\sum_{s \in S} p_s\int_\C
\vc^\top\tilde{\vg}_s\vy_s[\bar\nu(\omega_s)](\dd\vc)+U(\vx,\vy,\bar
\mu).\label{intercond1}
\end{align}
Now observe that the complementary slackness condition (iv) implies
$-\bar\vlambda^\top A \vx\geq -\bar\vlambda^\top \vb =
-\bar\vlambda^\top A\vx^*$, as $A\vx\leq\vb$ holds for all $\vx\in
X$, and the Lagrange multiplier $\bar \vlambda$ is nonnegative.
Similarly, as $T_s\vx+W_s\vy_s \ge \vh_s$ holds for all $\vy_s \in
\Y(\vx,\vxi(\omega_s)),~s\in S,$ and the Lagrange multiplier $\bar
\vpi_s$ is nonnegative for all $s\in S$, the complementary slackness
condition (v) implies that $\bar\vpi_s^\top(T_s\vx+W_s\vy_s)\geq
\bar\vpi_s^\top \vh_s=\bar\vpi^\top (T_s\vx^*+W_s\vy^*_s)$.
Therefore, the following relation is valid for any primal feasible
solution $(\vx,\vy)$:
\begin{equation}-\bar\vlambda^\top A \vx+\sum_{s \in S}
p_s\bar\vpi^\top_s(T_s\vx+W_s\vy_s)\geq -\bar\vlambda^\top A
\vx^*+\sum_{s \in S}
p_s\bar\vpi^\top_s(T_s\vx^*+W_s\vy^*_s).\label{intercond2}\end{equation}Observing
that $U(\vZ,\bar \mu)$ is independent of the decision vectors $\vx$
and $\vy$, by \eqref{intercond1} and \eqref{intercond2}, it is
sufficient to show that the inequality $-\E\left(\int_\C
\vc^\top\bar{\vg}\vx\,\bar\nu(\dd\vc)\right) -\sum_{s \in S}
p_s\int_\C
\vc^\top\tilde{\vg}_s\vy_s[\bar\nu(\omega_s)](\dd\vc)+U(\vx,\vy,\bar
\mu)\geq$ $-\E\left(\int_\C
\vc^\top\bar{\vg}\vx^*\,\bar\nu(\dd\vc)\right) -\sum_{s \in S}
p_s\int_\C
\vc^\top\tilde{\vg}_s\vy^*_s[\bar\nu(\omega_s)](\dd\vc)+U(\vx^*,\vy^*,\bar
\mu)$ is valid for any primal feasible solution $(\vx,\vy)$ (and
equality holds for $(\vx^*,\vy^*)$). Accordingly, we prove that the
following chain of inequalities holds for any $(\vx,\vy) \in
(\X,\Y)$, and that all inequalities hold with equality for
$(\vx,\vy)=(\vx^*,\vy^*)$.\small{\begin{align} &-\E\left(\int_\C
\vc^\top\bar{\vg}\vx\,\bar\nu(\dd\vc)\right) -\sum_{s \in S}
p_s\int_\C
\vc^\top\tilde{\vg}_s\vy_s[\bar\nu(\omega_s)](\dd\vc)+U(\vx,\vy,\bar
\mu) \label{fchain1}\\=& -\sum_{s \in S} p_s\int_\C
\vc^\top(\bar{\vg}_s\vx+\tilde{\vg}_s\vy_s)[\bar\nu(\omega_s)](\dd\vc)
\notag\\ +& \int\limits_{\C} \left(\var_\alpha(\vc^\top
\hat{\vG}(\vx,\vy))+\frac{1}{1-\alpha}\sum\limits_{s\in
S}p_s[\vc^\top\hat{\vg}_s(\vx,\vy_s)-\var_\alpha(\vc^\top
\hat{\vG}(\vx,\vy))]_+\right)\,\bar \mu(\dd(\vc))\label{fchain2}\\=&
-\sum_{s \in S} p_s\int_\C
\left(\vc^\top\hat{\vg}_s(\vx,\vy_s)-\var_\alpha(\vc^\top
\hat{\vG}(\vx,\vy))\right)[\bar\nu(\omega_s)](\dd\vc) \notag\\ +&
\frac{1}{1-\alpha}\sum\limits_{s\in S}p_s\int\limits_{\C}
[\vc^\top\hat{\vg}_s(\vx,\vy_s)-\var_\alpha(\vc^\top
\hat{\vG}(\vx,\vy))]_+\,\bar \mu(\dd(\vc))\label{fchain3}\\\geq &
\sum_{s \in S} p_s\int_\C
\left([\vc^\top\hat{\vg}_s(\vx,\vy_s)-\var_\alpha(\vc^\top
\hat{\vG}(\vx,\vy))]_+\right.\notag\\&\left.-(\vc^\top\hat{\vg}_s(\vx,\vy_s)-\var_\alpha(\vc^\top
\hat{\vG}(\vx,\vy)))\right)[\bar\nu(\omega_s)](\dd\vc)\label{fchain4}
\\&
\geq 0.\label{fchain5}
\end{align}}We note that an alternative definition
 of CVaR for a random variable $V$ is given by $\cvar_{\alpha}(V)=\var_\alpha(V) + \dfrac{1}{1-\alpha} \E([V-\var_\alpha(V)]_+) $, where $[z]_+ = \max(z, 0)$ \cite[]{Rockafellar00}.
Substituting this definition  into the third term and expanding the
expected value in the first term \eqref{fchain1} becomes
\eqref{fchain2}. Using the dual feasibility condition
\eqref{dualfeascond1} to replace $\bar \mu$ in the second term of
\eqref{fchain2} provides \eqref{fchain3}. Then, \eqref{fchain4}
follows from \eqref{dualfeascond2}; equality for
$(\vx,\vy)=(\vx^*,\vy^*)$ is ensured by the complementary slackness
condition (iii). Finally, \eqref{fchain5} is a consequence of the
trivial inequality $[a]_+\geq a$ and the non-negativity of the
random measure $\nu$. Equality is again ensured for
$(\vx,\vy)=(\vx^*,\vy^*)$, since the complementary slackness
condition (ii) implies that
$\vc^\top\hat{\vg}_s(\vx,\vy_s)-\var_{\alpha}(\vc^\top\hat{\vG}(\vx,\vy))\geq
0$ holds on the support of $\bar \nu(\omega_s)$ for all $s\in S$.

Finally, let us consider a primal optimal solution $(\vx^*,\vy^*)$
and a dual optimal solution $\mu^*$. We just showed above that there
exists some dual feasible $\bar \mu $ such that $U(\vx^*,\vy^*,\bar
\mu)=U(\vZ,\bar\mu)$ and $\cL(\vx^*,\vy^*,\bar
\mu)=\min_{(\vx,\vy)\in (\X,\Y)}\cL(\vx,\vy,\bar
\mu)=\vf^\top\vx^*+\sum_{s \in S} p_s \vq^\top_s \vy^*_s$. As the
set of feasible primal solutions is compact and the Lagrangian
function $\cL$ is continuous, there also exists some $(\hat\vx,\hat
\vy)$ such that $\cL(\hat\vx,\hat\vy,\mu^*)=\min_{(\vx,\vy)\in
(\X,\Y)}\cL(\vx,\vy,\mu^*)$. Then, by the optimality of $\mu^*$, we
have $\vf^\top\vx^*+\sum_{s \in S} p_s \vq^\top_s
\vy^*_s=\min_{(\vx,\vy)\in (\X,\Y)}\cL(\vx,\vy,\bar \mu)\leq
\min_{(\vx,\vy)\in (\X,\Y)}\cL(\vx,\vy,
\mu^*)=\cL(\hat\vx,\hat\vy,\mu^*)$, and consequently, the following
relation holds: \begin{align}&U(\vx^*,\vy^*,
\mu^*)-U(\vZ,\mu^*)=\cL(\vx^*,\vy^*,\mu^*) -\vf^\top\vx^*-\sum_{s
\in S} p_s \vq^\top_s \vy^*_s\notag\\&\geq
\cL(\vx^*,\vy^*,\mu^*)-\cL(\hat\vx,\hat\vy,\mu^*)=\cL(\vx^*,\vy^*,\mu^*)-\min_{(\vx,\vy)\in
(\X,\Y)}\cL(\vx,\vy,\mu^*)\geq0.\notag\end{align}On the other hand,
by the primal feasibility of $(\vx^*,\vy^*)$ and the non-negativity
of $\mu^*$, $U(\vx^*,\vy^*, \mu^*)-U(\vZ,\mu^*)\leq 0$ holds, which
immediately implies \eqref{lag-cond1} and \eqref{lag-cond2}.

\end{proof}

\bibliographystyle{spmpscinat}      
\bibliography{multivariate}
\end{document}